\documentclass{article}
\usepackage[utf8]{inputenc}

\usepackage{float}
\usepackage{amsmath,amsfonts,amssymb,relsize,geometry}
\usepackage{subcaption}
\usepackage{amsthm,xcolor}
\usepackage[round,authoryear]{natbib}

\usepackage{tikz}
\usetikzlibrary{shapes,arrows,automata}

\newcommand{\Cov}{\text{Cov}}
\newcommand{\M}{\mathcal{M}}

\newtheorem{thm}{Theorem}[section]
\newtheorem{prop}[thm]{Proposition}
\newtheorem{cor}{Corollary}[section]
\newtheorem{lem}[thm]{Lemma}
\theoremstyle{definition}
\newtheorem{dfn}{Definition}[section]

\DeclareMathOperator{\pa}{pa}
\DeclareMathOperator{\ch}{ch}
\DeclareMathOperator{\an}{an}
\DeclareMathOperator{\de}{de}
\DeclareMathOperator{\nd}{nd}
\DeclareMathOperator{\mb}{mb}

\DeclareMathOperator{\dis}{dis}

\newcommand{\B}{\mathcal{B}}
\newcommand{\D}{\mathcal{D}}

\theoremstyle{definition}
\newtheorem{ex}{Example}[section]

\title{Regression Identifiability and Edge Interventions in Linear Structural Equation Models}
\author{Bohao Yao, Robin J. Evans}
\date{}

\begin{document}

\maketitle

\begin{abstract}
    In this paper, we introduce a new identifiability criteria for linear structural equation models, which we call regression identifiability. We provide necessary and sufficient graphical conditions for a directed edge to be regression identifiable. Suppose $\Sigma^*$ corresponds to the covariance matrix of the graphical model $G^*$ obtained by performing an edge intervention to $G$ with corresponding covariance matrix $\Sigma$. We first obtain necessary and sufficient conditions for $\Sigma^*$ to be identifiable given $\Sigma$. Using regression identifiability, we obtain necessary graphical conditions for $\Sigma^*$ to be identifiable given $\Sigma$. We also identify what would happen to an individual data point if there were such an intervention. Finally, we provide some statistical problems where our methods could be used, such as finding constraints and simulating interventional data from observational data.
\end{abstract}

\section{Introduction}
One of the main goals in empirical science is to identify a causal effect by experimentation and analysis. In practice, however, it is extremely difficult to perform a completely randomised trial, where the data collected is free from confounding and selection bias. Hence, sophisticated methods and techniques of statistical data analysis that can identify causal effects from observational data are highly desirable. One such method which has become increasingly popular over the past decades is the \emph{structural equation model} (SEM).

The idea of SEMs originated from the seminal works of \citet{wright1921correlation}, who developed path
analysis to analyze the genetic makeup of offspring of laboratory animals. SEMs were later applied to econometrics by \citet{haavelmo1943statistical} and the social sciences by \citet{blalock2017causal}. In the present day, SEMs are applied in subject fields as varied as epidemiology, ecology, behavioural sciences, social sciences and economics \citep{grace2006structural, hershberger2003growth, rothman2005causation, sobel2000causal}. The reader is referred to \citet{bollen1989structural,pearl2000causality,spirtes2000causation} for a more thorough background on SEMs.


In Wright's work, the SEM was represented by a \emph{directed acyclic graph} (DAG), where each vertex represents a random variable and each edge represents a `direct effect'. This graph-based framework was later formalised by \citet{pearl2000causality}. When dealing with DAGs containing hidden variables, it is common to apply a \emph{latent projection} operation to produce a new family of graphs with only the observable variables \citep{pearl1992statistical}. This new family of graphs are known as the \emph{acyclic directed mixed graphs} (ADMGs). An example of this operation is given in Figure \ref{fig: intro-med}, with the DAG given on the left and the ADMG on the right.

\subsection{Linear Structural Equation Models}
Given an ADMG $G$ with vertex set $V$, a set of directed edges, $\D$, and a set of bidirected edges, $\B$, a linear SEM associated with $G$ is defined as
\begin{align}
\label{eqn: SEM}
    X_i=\sum\limits_{j\in\pa(i)}\lambda_{ji}X_j+\epsilon_i,\quad i\in V,
\end{align}
where $\pa(i)$ represents the set of parents of the vertex $i$, each $\lambda_{ji}$ is the edge coefficient of $j\to i$ and $\epsilon=(\epsilon_i)$ is a multivariate Gaussian vector with $\omega_{ij}:=\Cov(\epsilon_i,\epsilon_j)$. This is non-zero only if $i=j$ or $i\leftrightarrow j\in\B$. 
Let $\Lambda=(\lambda_{ij})$ be the matrix holding the edge coefficients, we can rewrite (\ref{eqn: SEM}) in matrix form
$$X=\Lambda^T X+\epsilon.$$

Since $G$ is acyclic, there is a topological ordering of vertices such that $\Lambda$ is a strictly upper triangular matrix. In particular, $I-\Lambda$ is invertible with determinant one. Hence, $X=(I-\Lambda)^{-T}\epsilon$ is the unique solution to the structural equations, where $X^{-T}$ represents inverse transpose of $X$. Let $\Omega=(\omega_{ij})=\Cov[\epsilon]$ be a covariance matrix of $\epsilon$. $X$ has the covariance matrix
\begin{align}
\label{eqn: main}
\Sigma:=\Cov[X]=(I-\Lambda)^{-T}\Omega(I-\Lambda)^{-1}.
\end{align}

\begin{ex}
    Consider the model in Figure \ref{fig: intro-med} where $X_4$ is a latent variable. Suppose the variables represent
    \begin{itemize}
        \item $X_1$ - Smoking frequency,
        \item $X_2$ - Tar in lungs,
        \item $X_3$ - Cough frequency,
        \item $X_4$ - Genes.
    \end{itemize}
    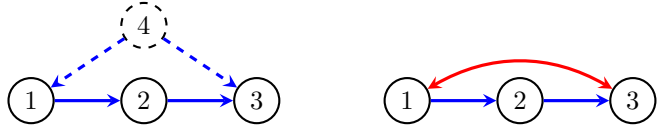
\begin{figure}
\centering
  \begin{tikzpicture}
  [rv/.style={circle, draw, thick, minimum size=6mm, inner sep=0.5mm}, node distance=15mm, >=stealth,
  hv/.style={circle, draw, thick, dashed, minimum size=6mm, inner sep=0.5mm}, node distance=15mm, >=stealth]
  \pgfsetarrows{latex-latex};
     \begin{scope}
  \node[rv]  (1)              {1};
  \node[rv, right of=1, yshift=0mm, xshift=0mm] (2) {2};
  \node[rv, right of=2, yshift=0mm, xshift=0mm] (3) {3};
  \node[hv, above of=2, yshift=-5mm, xshift=0mm] (4) {4};
  \draw[->, very thick, color=blue] (1) -- (2);
  \draw[->, very thick, color=blue] (2) -- (3);
  \draw[->, very thick, color=blue, dashed] (4) -- (3);
  \draw[->, very thick, color=blue, dashed] (4) -- (1);
  \end{scope}
  \begin{scope}[xshift = 5cm]
  \node[rv]  (1)              {1};
  \node[rv, right of=1, yshift=0mm, xshift=0mm] (2) {2};
  \node[rv, right of=2, yshift=0mm, xshift=0mm] (3) {3};
  \draw[->, very thick, color=blue] (1) -- (2);
  \draw[->, very thick, color=blue] (2) -- (3);
  \draw[<->, very thick, color=red] (1) to[bend left] (3);
  \end{scope}
     \end{tikzpicture}
 \caption{A graphical model (left) with its latent projection over $X_1,X_2,X_3$ (right)}
 \label{fig: intro-med}
\end{figure}
    The latent projection on the right can be represented as the following linear SEM,
    \begin{align*}
        &X_1=\epsilon_1,\\
        &X_2=\lambda_{12}X_1+\epsilon_2,\\
        &X_3=\lambda_{23}X_2+\epsilon_3,\\
        &\Cov(\epsilon_1,\epsilon_2)=\Cov(\epsilon_2,\epsilon_3)=0,\\
        &\Cov(\epsilon_1,\epsilon_3)=\omega_{13}.
    \end{align*}
\end{ex}

\subsection{Identifiability}

Identifiability in SEMs is a topic with a long history. A review of the classical conditions of this property, without taking the graphical structural into account, can be found in \citet{bollen1989structural}. We shall start with the definition of \emph{global identifiability}.

\begin{dfn}
Let $\phi:\Theta\to N$ be a rational map defined everywhere on the parameter space $\Theta$ into the natural parameter space $N$ of an exponential family. The model $\M=\mathrm{im}\ \phi$ is said to be \emph{globally identifiable} if $\phi$ is a one-to-one map on $\Theta$.
\end{dfn}

Sufficient graphical conditions for a linear SEM to be globally identifiable were first found by \citet{mcdonald2002can} and \citet{richardson2002ancestral}. 
The necessary and sufficient graphical conditions for linear SEMs to be globally identifiable were later given by \citet{drton2011global}. 

In standard literature, however, `identifiable' is usually applied to the identifiability criterion encountered in the instrumental variable model. This is often to referred to as \emph{almost-everywhere} identifiability or \emph{generic identifiability}. In particular, for any semialgebraic set $V\subset\Theta$, that is a subset defined by the zeroes of polynomial equations (of the form $\{x\mid f(x)=0\}$) and inequalities (of the form $\{x\mid f(x)>0\})$ over the reals. If $V\neq\Theta$, we have $\dim(V)<\dim(\Theta)$ \citep{cox2015ideals} and hence, $V$ has measure zero with respect to the Lebesgue measure \citep{okamoto1973distinctness} on $\Theta$.

\begin{dfn}
Let $\phi:\Theta\to N$ be a rational map defined everywhere on the parameter space $\Theta$ into the natural parameter space $N$ of an exponential family. The model $\M=\mathrm{im}\ \phi$ is said to be \emph{generically identifiable} if $\phi^{-1}(\phi(\theta))=\{\theta\}$ for almost all $\theta\in\Theta$ with respect to the Lebesgue measure.
\end{dfn}

From the definitions, we see that all globally identifiable models are also generically identifiable, but Example \ref{ex: identifiable} shows the converse generally fails.

\begin{figure}
\centering
  \begin{tikzpicture}
  [rv/.style={circle, draw, thick, minimum size=6mm, inner sep=0.5mm}, node distance=15mm, >=stealth,
  hv/.style={circle, draw, thick, dashed, minimum size=6mm, inner sep=0.5mm}, node distance=15mm, >=stealth]
  \pgfsetarrows{latex-latex};
  \node[rv]  (1)              {1};
  \node[rv, right of=1, yshift=0mm, xshift=0mm] (2) {2};
  \node[rv, right of=2, yshift=0mm, xshift=0mm] (3) {3};
  \draw[->, very thick, color=blue] (1) -- (2);
  \draw[->, very thick, color=blue] (2) -- (3);
  \draw[<->, very thick, color=red] (2) to[bend left] (3);
     \end{tikzpicture}
 \caption{The instrumental variable model.}
 \label{fig: IV}
\end{figure}
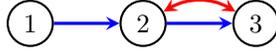

\begin{ex}
\label{ex: identifiable}
    Consider the instrumental variable model in Figure \ref{fig: IV}. We can recover the $\Lambda$ from the observed $\Sigma$ using $$\lambda_{12}=\frac{\sigma_{12}}{\sigma_{11}},\qquad\lambda_{23}=\frac{\sigma_{13}}{\sigma_{12}}.$$
    Note that the first denominator is always positive since $\Sigma$ is positive definite and the second denominator is zero if and only if $\lambda_{12}=0$. In particular, the map $\phi$ is injective precisely on the set $\{\Theta\mid\lambda_{12}\neq 0\}$. Hence, the instrumental variable model is generically identifiable but not globally identifiable.
\end{ex}

In this paper, we define a stronger identifiability criterion called \emph{regression identifiable}. In particular, if all the parameters in $\Lambda$ and $\Omega$ are regression identifiable, then the graphical model is also globally identifiable. Sufficient and necessary conditions for regression identifiability are given in Theorems \ref{thm: lambda_ab} and \ref{thm: omega_ab}.


\subsection{Interventions}
Variable interventions have been the main focus for the majority of causal inference literature. Utilising the frameworks established by \citet{neyman1932}, \citet{rubin1974estimating} and \citet{pearl2000causality}, variable intervention has been applied to various fields in empirical and social sciences. The interested reader is referred to works by \citet{halpern2005causes,tian2013causal,woodward2001causation} for a survey.

There has been a shift towards a more general notion of intervention in the past two decades, with \citet{korb2004varieties} proposing several generalisations of interventions. Similar ideas were introduced by \citet{eberhardt2007interventions}, where the shortcomings for each intervention were discussed. \citet{malinsky2018intervening} proposed an extension of the interventionist framework to enable the exploration of the consequences of intervention on a ``macro-level". A type of edge intervention in social networks has been proposed by \citet{ogburn2017causal} to study changes in network ties.Furthermore, \citet{shpitser2016causal} considered edge intervention in the context of mediation analyses.
In addition, \citet{robins2010alternative} proposed hypothetical randomised trials that are essentially a form of edge intervention. These ideas were later formalised as conditional separable effects by \citet{stensrud2022conditional}.
Most recently, \citet{sherman2020intervening} evaluated general intervention on network ties, which allows us to envision the counterfactual world where an edge is removed or added. The works by \citet{shpitser2016causal} and \citet{sherman2020intervening} are mainly studies of non-parametric models. 

The difference between edge and vertex interventions are highlighted in the example below.
\begin{ex}
Consider a dynamic treatment model introduced in \citet{robins1986new} with outcome $X_4$ and two treatment variables $X_1$ and $X_3$, where the second treatment is dependent on both the first treatment and an intermediate outcome $X_2$. The final outcome $X_4$ is dependent on the two treatments, the intermediate outcome, and an unobserved confounder between the two outcomes. The corresponding graphical model is shown in Figure \ref{fig: Treatment}. 

\begin{figure}
\centering
\begin{subfigure}{0.48\textwidth}
  \begin{tikzpicture}
  [rv/.style={circle, draw, thick, minimum size=6mm, inner sep=0.5mm}, node distance=15mm, >=stealth,
  hv/.style={rectangle, draw, thick, minimum size=6mm, inner sep=0.5mm}, node distance=15mm, >=stealth]
  \pgfsetarrows{latex-latex};
  \node[rv]  (1)              {1};
  \node[rv, right of=1, yshift=0mm, xshift=0mm] (2) {2};
  \node[rv, right of=2, yshift=0mm, xshift=0mm] (3) {3};
  \node[rv, right of=3] (4) {4};
  \draw[->, very thick, color=blue] (1) -- (2);
  \draw[->, very thick, color=blue] (2) -- (3);
  \draw[->, very thick, color=blue] (3) -- (4);
  \draw[->, very thick, color=blue] (1) to[bend right] (3);
  \draw[->, very thick, color=blue] (2) to[bend right] (4);
  \draw[->, very thick, color=blue] (1) to[bend right=45] (4);
  \draw[<->, very thick, color=red] (2) to[bend left] (4);
  \end{tikzpicture}
  \caption{}
  \label{fig: Treatment}
  \end{subfigure}
  \begin{subfigure}{0.48\textwidth}
  \begin{tikzpicture}
  [rv/.style={circle, draw, thick, minimum size=6mm, inner sep=0.5mm}, node distance=15mm, >=stealth,
  hv/.style={rectangle, draw, thick, minimum size=6mm, inner sep=0.5mm}, node distance=15mm, >=stealth]
  \pgfsetarrows{latex-latex};
  \node[rv]  (1)              {1};
  \node[rv, right of=1, yshift=0mm, xshift=0mm] (2) {2};
  \node[hv, right of=2, yshift=0mm, xshift=0mm] (3) {3};
  \node[rv, right of=3] (4) {4};
  \draw[->, very thick, color=blue] (1) -- (2);
  \draw[->, very thick, color=blue] (3) -- (4);
  \draw[->, very thick, color=blue] (2) to[bend right] (4);
  \draw[->, very thick, color=blue] (1) to[bend right=45] (4);
  \draw[<->, very thick, color=red] (2) to[bend left] (4);
  \end{tikzpicture}
  \caption{}
  \label{fig: VIntervention}
  \end{subfigure}
  \caption{Double treatment model (left) with a vertex intervention (right).}
\end{figure}
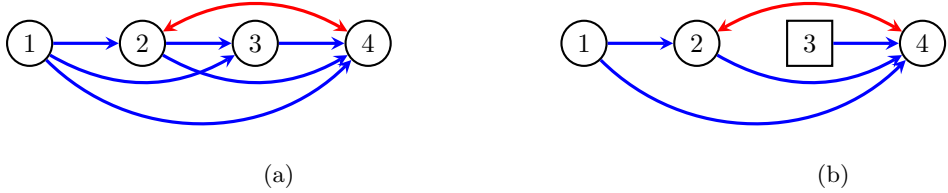
\begin{itemize}
    \item {\bf Vertex intervention}: Suppose we intervene on $X_3$ in order to provide the second treatment to all the patients. This is graphically equivalent replacing the random variable $X_3$ with a fixed variable and removing all incoming edges into the vertex 3 as shown in Figure \ref{fig: VIntervention}.
    
    Note that not all interventional distributions are identifiable (e.g. if we intervene on the intermediate outcome $X_2$) as not all vertices are \emph{fixable}. We will define fixable vertices rigorously in Section 2.
    \item {\bf Edge intervention}: Edge interventions involve the removal or the change in strength of direct causal effects. For instance, we may want to change the strength of the direct causal effect of $X_2$ on $X_3$ due to a shortage or surplus of available treatments.
\end{itemize}
\end{ex}

We now provide an example from the social sciences where edge interventions could be applied.

\begin{figure}
  \begin{center}
  \begin{tikzpicture}
  [rv/.style={circle, draw, thick, minimum size=6mm, inner sep=0.5mm}, node distance=15mm, >=stealth,
  hv/.style={circle, draw, thick, dashed, minimum size=6mm, inner sep=0.5mm}, node distance=15mm, >=stealth]
  \pgfsetarrows{latex-latex};
  \begin{scope}
  \node[rv]  (1)            {$Y_{1,1}$};
  \node[rv, right of=1] (2) {$Y_{1,2}$};
  \node[rv, right of=2] (3) {$Y_{1,3}$};
  \node[rv, below of=1] (4) {$Y_{2,1}$};
  \node[rv, below of=2] (5) {$Y_{2,2}$};
  \node[rv, below of=3] (6) {$Y_{2,3}$};
  \draw[->, very thick, color=blue] (1) -- (2);
  \draw[->, very thick, color=blue] (2) -- (3);
  \draw[->, very thick, color=red] (4) -- (5);
  \draw[->, very thick, color=red] (5) -- (6);
  \draw[->, very thick, color=blue] (1) to[bend left] (3);
  \draw[->, very thick, color=red] (4) to[bend right] (6);
  \draw[->, very thick, color=blue] (1) -- (5);
  \draw[->, very thick, color=blue] (1) -- (6);
  \draw[->, very thick, color=blue] (2) -- (6);
  \draw[->, very thick, color=red] (4) -- (2);
  \draw[->, very thick, color=red] (4) -- (3);
  \draw[->, very thick, color=red] (5) -- (3);
  \end{scope}
  \begin{scope}[xshift=6.5cm]
  \node[rv]  (1)            {$Y_{1,1}$};
  \node[rv, right of=1] (2) {$Y_{1,2}$};
  \node[rv, right of=2] (3) {$Y_{1,3}$};
  \node[rv, below of=1] (4) {$Y_{2,1}$};
  \node[rv, below of=2] (5) {$Y_{2,2}$};
  \node[rv, below of=3] (6) {$Y_{2,3}$};
  \draw[->, very thick, color=blue] (1) -- (2);
  \draw[->, very thick, color=blue] (2) -- (3);
  \draw[->, very thick, color=red] (4) -- (5);
  \draw[->, very thick, color=red] (5) -- (6);
  \draw[->, very thick, color=blue] (1) to[bend left] (3);
  \draw[->, very thick, color=red] (4) to[bend right] (6);
  \draw[->, very thick, color=blue] (1) -- (5);
  \draw[->, very thick, color=red] (4) -- (2);
  \end{scope}
    \end{tikzpicture}
 \caption{(a) A DAG representing time series cross sectional data on two countries where country 2 has a trade agreement with country 1; (b) the DAG in (a) after an intervention is performed, severing the alliance between countries 1 and 2 at $t= 2$.}
  \label{fig: Trade}
  \end{center}
\end{figure}
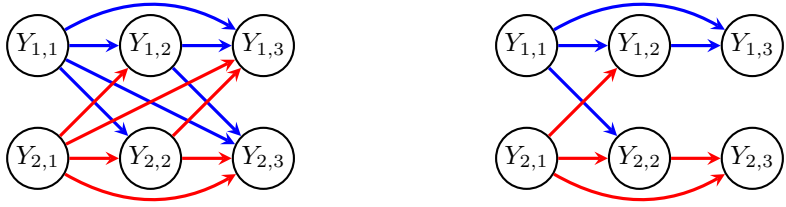
\begin{ex}
\label{ex: edge-intervention}
Consider Figure \ref{fig: Trade}, a model of trade relations between countries by \citet{sherman2020intervening}. Each country $i$ is represented by temporally sequential observations $\{Y_{i,t}\mid t\in\{1,2,\dots\}\}$ where each $Y$ is a vector of economic variables (e.g. GDP). We could consider a move from Figure \ref{fig: Trade} (a) to (b) by a severance of trade relations between two countries. Similarly, we could consider a reverse intervention of extending a trade agreement between countries, moving from Figure \ref{fig: Trade} (b) to (a).
Once again, we see that these interventions could not simply be performed by an intervention of the variables $Y_{1,3}$ and $Y_{2,3}$ as these will remove all incoming arrows into $Y_{1,3}$ and $Y_{2,3}$. 
\end{ex}


In this paper, we study edge interventions in the parametric case where the model is a linear SEM. Sufficient graphical conditions for a distribution to be identifiable after an edge intervention will be given in Corollary \ref{cor: remove-directed}.

\subsubsection*{Effects of Interventions}

Questions relating to the effects of interventions are central to various topics in scientific research and policy making. For instance, if a patient has died due to a heart transplant, would he be alive if the heart transplant had not happened?

In addition to studying when an interventional distribution is identifiable, finding the effect of an intervention on a data point has important practical applications. In Example \ref{ex: edge-intervention}, an economic advisor could assess the effects of a particular foreign policy or conflict will have on the economy.

Popular methods used to estimate the total causal effect of a particular treatment on the outcome include adjusting for confounding \citep{shpitser2010validity,perkovic2018complete} and inverse probability weighting \citep{robins2000marginal}.
In the medical literature, common definitions of causal effect are the controlled direct effect \citep{robins1992identifiability}, the principal stratum effect \citep{frangakis2002principal} and the conditional separable effect \citep{stensrud2022conditional}. In particular, conditional separable effect could be interpreted as the effect of an edge intervention.

In this paper, we study how an edge intervention would affect individual data points using regression identifiability under the assumption that the error terms are stable. This is possible due to the linearity of error terms.

\subsection{Outline of the paper}

In this paper, we will study the effects of edge interventions in Gaussian ADMGs. In particular, we shall study how adding or removing a directed edge will affect the distribution of the model and provide sufficient graphical conditions to identify the new distribution and (under additional assumptions) the counterfactual for individual-level data.

This paper is structured as follows: in Section 2, we will provide the preliminaries for ADMGs. 
In Section 3, we will study the effects of removing a directed edge. In particular, we provide sufficient algebraic and graphical conditions for $\Sigma^*$ to be identifiable given $\Sigma$. To achieve the graphical conditions, we will introduce a stronger identifiability condition which we call \emph{regression identifiability}. Necessary and sufficient conditions for an edge and a directed path to be regression identifiable are also given. We will also identify the effects of removing an edge on individual data points.

In Section 4, we will consider the reverse and study the effects of adding a directed edge. Once again, we provide the sufficient algebraic conditions and graphical conditions for $\Sigma^*$ to be identifiable given $\Sigma$. We will also identify the effects of adding an edge on individual data points.
In Section 5, we will discuss the limitations of our graphical conditions.

\section{Preliminaries}
In this section, we will provide some standard definitions of graphical models. The confident reader may skip this section.
\subsection{Graphical Models}
A \emph{directed mixed graph} is a triple $G=(V,\D,\B)$ where $V$ is a set of vertices, $\D$ is the set of directed edges ($\rightarrow$) and $\B$ is the set of bidirected edges ($\leftrightarrow$).
Edges in $\D$ are oriented whereas edges in $\B$ have no orientation.
A \emph{loop} is an edge joining a vertex to itself.
In this paper, we will only consider graphs without any loops, that is $i\to i\notin\D$ and $i\leftrightarrow i\notin\B$. 

A \emph{path} is a walk where all vertices are distinct.
A \emph{directed path} of length $\ell$ is a path of the form $v_0\rightarrow v_1\rightarrow \dots\rightarrow v_\ell$. Similarly, a \emph{bidirected path} of length $\ell$ is a path of the form $v_0\leftrightarrow v_1\leftrightarrow \dots\leftrightarrow v_\ell$.
Adding an additional edge $v_\ell\to v_0$ to a directed path, we obtain a \emph{directed cycle}.
An \emph{acyclic directed mixed graph} (ADMG) is a directed mixed graph without any directed cycles. 
A \emph{directed acyclic graph} (DAG) is an ADMG without any bidirected edges. 

Suppose $x,y\in V$. If $x\rightarrow y$, we say $x$ is a \emph{parent} of $y$ and $y$ is a \emph{child} of $x$. If either $x=y$ or there is a directed path from $x$ to $y$, we say that $x$ is an \emph{ancestor} of $y$ and $y$ is a \emph{descendant} of $x$. Otherwise, we say that $y$ is a \emph{non-descendant} of $x$. The sets of all parents, children, ancestors, descendants and non-descendants of $x$ are denoted $\pa(x)$, $\ch(x)$, $\an(x)$, $\de(x)$ and $\nd(x)$ respectively. We can also extend these definitions to a vertex set $S\subseteq V$, e.g. $\an(S)=\bigcup_{v\in S}\an(v)$. 

A graph, $G'=(V',E')$, is a \emph{subgraph} of $G=(V,E)$ if $V'\subseteq V$ and $E'\subseteq E$.
If $E'$ contains all the edges of $G$ that has both endpoints in $V'$, then $G'$ is called an \emph{induced subgraph} of $G$.
We denote the induced subgraph on the vertex set $W\subseteq V$ by $G_W$.

A set of vertices $S\subseteq V$ is \emph{bidirected-connected} if for every $u,v\in S$, there is a bidirected path from $u$ to $v$ in $G_S$. A maximal bidirected-connected set of vertices is called a \emph{district}. For $v\in V$, let $\dis(v)$ denote the district containing $v$ in $G$. The \emph{Markov blanket} of $v$ is $$\mb(v):=\pa(\dis(v))\cup(\dis(v)\backslash\{v\}).$$
We write $\dis_S(v)$ for the district containing $v$ in $G_S$. Similarly, we define $\mb_S(v)$ to be the Markov blanket of $v$ in $G_S$.
A vertex $v$ is \emph{fixable} if $\de(v)\cap\dis(v)=\{v\}$ \citep{richardson2017nested}. Fixable vertices are crucial as we are always able to identify the causal effect of $v$ on any $d\in\de(v)$.

Given a path $\pi$, a vertex $v$ is a \emph{collider on $\pi$} if two arrowheads of $\pi$ meet head to head at $v$ (i.e. if $\pi$ contains a subpath of the form: $\rightarrow v\leftarrow$, $\leftrightarrow v\leftarrow$, $\rightarrow v\leftrightarrow$ or $\leftrightarrow v\leftrightarrow$). Otherwise, $v$ is a \emph{non-collider on $\pi$}. 
\subsection{Treks}
A \emph{trek} from $i$ to $j$ is a walk from $i$ to $j$ without any colliders. Therefore, all treks are of the form
$$v^L_\ell\leftarrow v^L_{\ell-1}\leftarrow\dots\leftarrow v^L_1\leftarrow v^L_0\leftrightarrow v^R_0\to v^R_1\to\dots\to v^R_{r-1}\to v^R_r$$ or
$$v^L_\ell\leftarrow v^L_{\ell-1}\leftarrow\dots\leftarrow v^L_1\leftarrow v_0^{LR}\to v^R_1\to\dots\to v^R_{r-1}\to v^R_r,$$
where $v^L_\ell=i$, $v^R_r=j$. Note that paths on either side may have length zero, so a trek may just be a directed path. For a trek $\pi$ with no bidirected edges and a source $i$, we define the trek monomial as $$\pi(\Lambda,\Omega)=\omega_{ii}\prod\limits_{x\to y\in\pi}\lambda_{xy}.$$
For a trek $\pi$ with a bidirected edge $i\leftrightarrow j$, we define the trek monomial as 
$$\pi(\Lambda,\Omega)=\omega_{ij}\prod\limits_{x\to y\in\pi}\lambda_{xy}.$$

\begin{thm}[Trek rule]
\label{thm: trek}
Let $\mathcal{T}_{vw}$ be the set of all treks from $v$ to $w$. The covariance matrix $\Sigma$ for an ADMG $G$ is given by
\begin{align}
\label{eqn: intro-trek-rule}
    \sigma_{vw}=\sum\limits_{\pi\in\mathcal{T}_{vw}}\pi(\Lambda,\Omega).
\end{align}
\end{thm}

The proof for Theorem \ref{thm: trek} originated from \cite{wright1934method}, and is obtained from expanding out the matrix in (\ref{eqn: main}) using Taylor series.

\section{Removing a Directed Edge}
\label{section: remove-directed-edge}

Suppose we have an ADMG $G=(V,\D,\B)$ with a corresponding distribution $N(0,\Sigma)$. Suppose we want to remove some directed edge $a^*\to b$. Let $A=\nd(b)$, $B=\{b\}$ and $C=\de(b)\backslash\{b\}$, so that $A\cup B\cup C=V$. Throughout this section, we suppose that after removing the directed edge $a^*\to b$, we obtain a new ADMG $G^*$ with a corresponding distribution $N(0,\Sigma^*)$.

\subsection{Effect on Covariance}
\label{section: directed-covariance}
First, we shall find an expression for each entry of $\Sigma^*$.
\begin{lem}
\label{lem_sigma_ab}
Every trek from some vertex $a\in A$ to $b$ containing the directed edge $a^*\to b$ must end with $\dots a^*\to b$.
\end{lem}
\begin{proof}
Since $a\in\nd(b)$, we do not have a directed path $a\leftarrow\dots\leftarrow b$. Hence, treks from $a$ to $b$ 
must end with either $\rightarrow$ or $\leftrightarrow$. 
Suppose there is a trek from $a$ to $b$ containing the edge $a^*\to b$ but does not end with $\dots a^*\to b$. Then the edge $a^*\to b$ must be contained in the middle of the trek. Since all treks from $a$ to $b$ ends with either $\rightarrow$ or $\leftrightarrow$, if $a^*\to b$ is on the left side of the trek, we can no longer reach $a$ since $a\in\nd(b)$.
If $a^*\to b$ is on the right side of the trek, we will end up with a directed cycle:
$$a\longleftarrow\dots a^*\longrightarrow b\longrightarrow\dots\longrightarrow b.$$
Hence, treks from $a$ to $b$ containing $a^*\to b$ must end with $\dots a^*\to b$.
\end{proof}

\begin{cor}
\label{cor: directed-sigma_ab}
For all $a\in A$, $\sigma_{ab}^*=\sigma_{ab}-\sigma_{aa^*}\lambda_{a^*b}$.
\end{cor}
\begin{proof}
By Lemma \ref{lem_sigma_ab}, all treks from $a$ to $b$ that contains the edge $a^*\to b$ are simply the treks from $a$ to $a^*$ with the additional edge $a^*\to b$ appended to the end. By the trek rule, the new covariance is equal to the old covariance with all the treks from $a$ to $b$ containing the edge $a^*\to b$ subtracted.
\end{proof}

\begin{lem}
Let $a,a'\in A$. There are no treks between $a$ and $a'$ containing the edge $a^*\to b$.
\end{lem}
\begin{proof}
Suppose there is a trek from $a$ to $a'$ containing the edge $a^*\to b$. Then the trek must be
$$a\longleftarrow\dots a^*\longrightarrow b\longrightarrow\dots\longrightarrow a'\quad\text{or}\quad a\longleftarrow\dots\longleftarrow b\longleftarrow a^*\dots\longrightarrow a'.$$
But we have $a,a'\in\nd(b)$ which is a contradiction.
\end{proof}

\begin{cor}
\label{cor: directed-sigma_aa}
$\Sigma^*_{AA}=\Sigma_{AA}$.
\end{cor}


\begin{lem}
\label{lem-removeedge-bb}
All treks from $b$ to $b$ containing the edge $a^*\to b$ must either begin with $b\leftarrow a^*$ on the left side or end with $a^*\to b$ on the right side (or both).
\end{lem}
\begin{proof}
Otherwise, the edge $a^*\to b$ is contained in the middle of the trek. Similar to the proof of Lemma \ref{lem_sigma_ab}, we will end up with a directed cycle:
$$b\longleftarrow\dots\longleftarrow b\longleftarrow a^*\dots\longrightarrow b\quad\text{or}\quad b\longleftarrow\dots a^*\longrightarrow b\longrightarrow\dots\longrightarrow b.$$
\end{proof}

\begin{cor}
\label{cor: directed-sigma_bb}
$\sigma_{bb}^* = \sigma_{bb}-2\lambda_{a^*b}\sigma_{a^*b}+\lambda_{a^*b}^2\sigma_{a^*a^*}$.
\end{cor}
\begin{proof}
Consider all treks from $b$ to $b$ containing $a^*\to b$. If a trek starts with $a^*\to b$ on the left side, this is just a trek from $a^*$ to $b$ with the additional edge $b\leftarrow a^*$ appended to the start. Similarly, a trek ending with $a^*\to b$ on the right side is a trek from $b$ to $a^*$ with the additional edge $a^*\to b$ appended at the end. A trek that starts and ends with $a^*\to b$ is a trek from $a^*$ to $a^*$ with $b\leftarrow a^*$ appended to the start and $a^*\to b$ appended to the end. The result follows from the inclusion-exclusion principle.
\end{proof}

\begin{lem}
Treks from $b$ to $c\in C$ containing the edge $a^*\to b$ must either start with $b\leftarrow a^*$ on the left side or contain $a^*\to b$ on the right side (or both).
\end{lem}
\begin{proof}
Similar to the proof of Lemma \ref{lem-removeedge-bb}, if a trek contains $b\leftarrow a^*$ on the left side but does not start with $b\leftarrow a^*$, we have a directed cycle:
$$b\longleftarrow\dots\longleftarrow b\longleftarrow a^*\dots\longrightarrow c.$$
\end{proof}

\begin{cor}
\label{cor: directed-sigma_bc}
For all $c\in C$, 
\begin{align}
\label{eqn: sigma_bc^*}
\sigma_{bc}^*=\sigma_{bc}-\lambda_{a^*b}\sigma_{a^*c}-\lambda_{a^*b}\sigma_{a^*b}\sigma(\mathcal{D}_{bc})+\lambda_{a^*b}^2\sigma_{a^*a^*}\sigma(\mathcal{D}_{bc}),
\end{align} 
where $\mathcal{D}_{bc}$ denotes all directed paths from $b$ to $c$ and $\sigma(\mathcal{D}_{bc})\in\mathbb{R}[\Sigma]$ is the corresponding rational function obtained by summing up the relevant trek covariances.
\end{cor}
\begin{proof}
Similar to the proof of Corollary \ref{cor: directed-sigma_bb}, treks starting with $b\leftarrow a^*$ are of the form:
$$b\longleftarrow a^*\dots\longrightarrow c$$
which corresponds to the second term of (\ref{eqn: sigma_bc^*}) while treks containing $a^*\to b$ on the right side are of the form:
$$b\longleftarrow\dots a^*\longrightarrow b\longrightarrow\dots\longrightarrow c$$
which corresponds to the third term of (\ref{eqn: sigma_bc^*}). Finally, treks containing $a^*\to b$ on both sides are of the form:
$$b\longleftarrow a^*\longleftarrow\dots\longrightarrow a^*\longrightarrow b\longrightarrow\dots\longrightarrow c$$
which corresponds the last term of (\ref{eqn: sigma_bc^*}). The result follows from the inclusion-exclusion principle.
\end{proof}

\begin{cor}
\label{cor: directed-sigma_cc}
For all $c,c'\in C$, 
\begin{align}
\label{sigma_cc^*}
\sigma_{cc'}^*=\sigma_{cc'}-\lambda_{a^*b}\sigma_{a^*c}\sigma(\mathcal{D}_{bc'})-\lambda_{a^*b}\sigma_{a^*c'}\sigma(\mathcal{D}_{bc})+\lambda_{a^*b}^2\sigma_{a^*a^*}\sigma(\mathcal{D}_{bc})\sigma(\mathcal{D}_{bc'}).
\end{align} 
\end{cor}
\begin{proof}
Similar to the proof of Corollary \ref{cor: directed-sigma_bc}, but the treks can contain $a^*\to b$ anywhere on either side. Note that the edge $a^*\to b$ can only be used at most once on each side of a trek as the graph is acyclic. The treks will be of the forms:
$$c\longleftarrow\dots a^*\longrightarrow b\longrightarrow\dots\longrightarrow c'$$
$$c\longleftarrow\dots\longleftarrow b\longleftarrow a^*\dots\longrightarrow c'$$
$$c\longleftarrow\dots\longleftarrow b\longleftarrow a^*\longleftarrow\dots\longrightarrow a^*\longrightarrow b\longrightarrow\dots\longrightarrow c'$$
corresponding to the second, third and last terms of (\ref{sigma_cc^*}) respectively.
\end{proof}

\begin{lem}
Treks from $a\in A$ to $c\in C$ containing the edge $a^*\to b$ must contain that edge on the right side. In particular, the source of all treks from $a$ to $c$ lies entirely within $A$.
\end{lem}

\begin{proof}
If we have a trek from $a$ to $c$ with $b$ on the left side, it will be of the form $a\leftarrow\dots\leftarrow b\dots$. But $a\in\nd(b)$.
\end{proof}

\begin{cor}
For all $a\in A$ and $c\in C$, $\sigma_{ac}^* =\sigma_{ac}-\sigma_{aa^*}\lambda_{a^*b}\sigma(\mathcal{D}_{bc})$.
\end{cor}

Since each entry of $\Sigma^*$ is a function of $\Sigma$, $\lambda_{a^*b}$ and $\sigma(\mathcal{D}_{bc})$, we obtain the following Corollary.

\begin{prop}
\label{prop: remove-directed}
Suppose that $N(0,\Sigma)$ is a distribution corresponding to the graphical model $G$ and the distribution $N(0,\Sigma^*)$ corresponds to the new graphical model $G^*$ obtained by removing the directed edge $a^*\to b$ from $G$. Then, $\Sigma^*$ is identifiable given $\Sigma$ if and only if both $\lambda_{a^*b}$ and $\sigma(\mathcal{D}_{bc})$ are identifiable for all $c\in C$.
\end{prop}

\begin{proof}
\begin{itemize}
    \item[($\Leftarrow$)] Follows from Corollaries \ref{cor: directed-sigma_ab}, \ref{cor: directed-sigma_aa}, \ref{cor: directed-sigma_bb}, \ref{cor: directed-sigma_bc} and \ref{cor: directed-sigma_cc} 
    \item[($\Rightarrow$)] From Corollary \ref{cor: directed-sigma_ab}, we have
\begin{align*}
    \sigma_{a^*b}^*&=\sigma_{a^*b}-\sigma_{a^*a^*}\lambda_{a^*b},\\
    \lambda_{a^*b}&=\frac{\sigma_{a^*b}-\sigma_{a^*b}^*}{\sigma_{a^*a^*}}.
\end{align*}
Since $\sigma_{a^*a^*}>0$, if $\sigma_{a^*b}^*$ is identifiable given $\Sigma$, $\lambda_{a^*b}$ is also identifiable. Similarly, we can rearrange (\ref{eqn: sigma_bc^*}) into 
$$\sigma(\mathcal{D}_{bc})=\frac{\sigma^*_{bc}-\sigma_{bc}+\lambda_{a^*b}\sigma_{a^*c}}{\lambda_{a^*b}\sigma_{a^*a^*}-\lambda_{a^*b}\sigma_{a^*b}}.$$
Since $\Sigma$ is positive definite, $\sigma(\mathcal{D}_{bc})$ is also identifiable.
\end{itemize}
\end{proof}


\subsection{Regression Identifiability}
As we have seen in Proposition \ref{prop: remove-directed}, in order for $\Sigma^*$ to be identifiable, both $\lambda_{a^*b}$ and $\sigma(\mathcal{D}_{bc})$ are identifiable for all $c\in C$. We first define the identifiability criterion we will be using in our paper.

\begin{dfn}
\label{dfn: identified-single-reg}
A directed edge $v\to w$ can be \emph{regression identifiable} if $\lambda_{vw}=\beta_{vw\cdot S}$ for some vertex subset $S\subseteq V\backslash\{v,w\}$, where $\beta_{vw\cdot S}=\sigma_{vw\cdot S}/\sigma_{vv\cdot S}$ is the regression coefficient of $X_v$ on $X_w$ conditional on $X_S$, and $\sigma_{vw\cdot S}=\sigma_{vw}-\Sigma_{vS}\Sigma^{-1}_{SS}\Sigma_{Sw}$.
\end{dfn}

The set $S$ in Definition \ref{dfn: identified-single-reg} is often referred to as the \emph{adjustment set} \citep{pearl2000causality} and criteria for such a set have been well studied \citep{shpitser2010validity,perkovic2015complete,perkovic2018complete}. While all edges that are regression identifiable are also globally identifiable, the converse is not true in general. We will discuss more about the limitations when considering identification by regression in Section \ref{section: limitations}. We now define our identifiability criterion for $\Sigma^*$

\begin{dfn}
Suppose that $N(0,\Sigma)$ is a distribution corresponding to the graphical model $G$ and the distribution $N(0,\Sigma^*)$ corresponds to the new graphical model $G^*$ obtained by removing the directed edge $a^*\to b$ from $G$. We say that $\Sigma^*$ is \emph{simply identifiable} given $\Sigma$ if both $\lambda_{a^*b}$ and $\sigma(\mathcal{D}_{bc})$ are regression identifiable for all $c\in C$.
\end{dfn}

By definition, if $\Sigma^*$ is simply identifiable given $\Sigma$, then $\Sigma^*$ is identifiable given $\Sigma$. We shall first discuss necessary and sufficient conditions for when the directed path $\sigma(\mathcal{D}_{bc})$ is identifiable by regression for all $c\in C$.

\begin{thm}
\label{thm: fixable-paths}
The path $\sigma(\mathcal{D}_{bc})$ is regression identifiable for some $c\in C$ if and only if $b$ is fixable in $G_{\an(c)}$. In particular, $\sigma(\mathcal{D}_{bc})=\beta_{bc\cdot S(c)}$ where $S(c)=\mb_{\an(c)}(b)$.
\end{thm}

\begin{proof}
\begin{itemize}
    \item [($\Leftarrow$)] It suffices to prove that $\sigma(\mathcal{D}_{bc})=\beta_{bc\cdot\eta}.$ Consider all paths from $b$ to $c$ conditioned on $\mb_{\an(c)}(b)$. First consider the case when the path starts with $b\to\dots$. If it is a directed path, the path is open because $\mb_{\an(c)}(b)$ does not contain any descendants of $b$. If there is a collider node $v$ on the path, the path is blocked if $v\notin\mb_{\an(c)}(b)$. But if all the collider nodes are in $\mb_{\an(c)}(b)$, then the first collider is a descendant of $b$. Hence there exists some vertex in the same district as $b$ that is also a descendant of $b$, contradicting the fact that $b$ is fixable.

    Now, all paths that start with $b\leftarrow a\dots$ are blocked since $a\in\pa(b)\subseteq\mb_{\an(c)}(b)$. 
    Paths starting with $b\leftrightarrow d_1\leftrightarrow\dots\leftrightarrow d_n\leftarrow a'\dots$ are blocked since $a'\in\pa(d_n)\subset\mb_{\an(c)}(b)$ and paths starting with $b\leftrightarrow d_1\leftrightarrow\dots\leftrightarrow d_n\rightarrow\dots$ are blocked since $d_n\in\mb_{\an(c)}(b)$. Note that $c\notin\dis(b)$ since $b$ is fixable in $G_{\an(c)}$.
    So the only open paths are the direct paths.

    \item[($\Rightarrow$)] Now, suppose that $b$ is not fixable in $G_{\an(c)}$. Hence, there exists some $c'\in\de(b)\cap\dis(b)\subseteq C$. Therefore, there is a bidirected path $b\leftrightarrow d_1\leftrightarrow\dots\leftrightarrow d_m\leftrightarrow c'$ from $b$ to $c'$. Since each $d_i\in\an(c)$, we have an open trek from $b$ to $c'$ of the form $b\leftarrow\dots\leftarrow d_1\to\dots\to c'$ or $b\leftrightarrow d_1\to\dots\to c'$. If we conditioned on either $d_1$ or any vertex in $\de(d_1)$, then we have the open trek $b\leftarrow\dots\leftarrow d_1\leftrightarrow d_2\to\dots\to c'$ or $b\leftrightarrow d_1\leftrightarrow d_2\to\dots\to c'$. Conditioning on $d_2$ or any vertex in $\de(d_1)$, we have the open trek $b\leftarrow\dots\leftarrow d_1\leftrightarrow d_2\leftrightarrow d_3\to\dots\to c'$ or $b\leftrightarrow d_1\leftrightarrow d_2\leftrightarrow d_3\to\dots\to c'$ and so on until either the trek $b\leftarrow\dots\leftarrow d_1\leftrightarrow \dots\leftrightarrow c'$ or $b\leftrightarrow \dots\leftrightarrow c'$ is open. Hence, $\sigma(\mathcal{D}_{bc'})$ is not regression identifiable.
\end{itemize}
\end{proof}

Next, we shall provide necessary and sufficient conditions for when the edge coefficient $\lambda_{a^*b}$ is identifiable by regression.



\begin{thm}
\label{thm: lambda_ab}
$\lambda_{ab}$ is regression identifiable if and only if
\begin{itemize}
    \item $a\notin\dis_{{\an(b)}}(b)$ and
    \item there is no directed edge from $a$ to $d$ for all $d\in\dis_{{\an(b)}}(b)$.
\end{itemize}
In particular, $\lambda_{ab}=\beta_{ab\cdot S(b)\backslash\{a\}}$ where $S(b)=\mb_{\an(b)}(b)$.
\end{thm}

\begin{proof}
\begin{enumerate}
    \item[($\Leftarrow$)] Consider all paths from $a$ to $b$ conditioned on $\mb_{\an(b)}(b)\backslash\{a\}$. In $G_{\an(b)}$, all paths ending with $\dots\to b$ are blocked unless they end with the edge $a\to b$. For $d_1,\dots,d_n\in\dis(b)$, all paths ending with $\leftarrow d_1\leftrightarrow\dots\leftrightarrow d_n\leftrightarrow b$ are blocked since $d_1\in\mb(b)$ and paths ending with $a'\to d_1\leftrightarrow\dots\leftrightarrow d_n\leftrightarrow b$ are blocked since $a'\in\mb(b)$ ($a'\neq a$ as there is no directed edge from $a$ to $d_i$). Paths ending with $\dots\leftarrow b$ do not exist in the induced subgraph $G_{\an(b)}$. Furthermore, there are no bidirected paths from $a$ to $b$ in the induced subgraph $G_{\an(b)}$ since $a\notin\dis_{{\an(b)}}(b)$.
    
    \item[($\Rightarrow$)] First, consider the case where $a\in\dis_{{\an(b)}}(b)$. Then there is a bidirected path $a\leftrightarrow d_1\leftrightarrow\dots\leftrightarrow d_m\leftrightarrow b$ from $a$ to $b$. Since each $d_i\in\an(b)$, we have an open trek from $a$ to $b$ of the form $a\leftarrow\dots\leftarrow d_1\to\dots\to b$ or $a\leftrightarrow d_1\to\dots\to b$. If we conditioned on either $d_1$ or any vertex in $\de(d_1)$, then we have the open trek $a\leftarrow\dots\leftarrow d_1\leftrightarrow d_2\to\dots\to b$ or $a\leftrightarrow d_1\leftrightarrow d_2\to\dots\to b$. Conditioning on $d_2$ or any vertex in $\de(d_1)$, we have the open trek $a\leftarrow\dots\leftarrow d_1\leftrightarrow d_2\leftrightarrow d_3\to\dots\to b$ or $a\leftrightarrow d_1\leftrightarrow d_2\leftrightarrow d_3\to\dots\to b$ and so on until either the trek $a\leftarrow\dots\leftarrow d_1\leftrightarrow \dots\leftrightarrow b$ or $a\leftrightarrow \dots\leftrightarrow b$ is open. Hence, $\lambda_{ab}$ cannot be identified by an ordinary regression.
    
    Now, suppose that we have a directed edge $a\to d_1$ for some $d_1\in \dis_{{\an(b)}}(b)$. Suppose there is a bidirected path $d_1\leftrightarrow d_2\leftrightarrow \dots\leftrightarrow b$ in $G_{\an(b)}$. Since $d_1\in\an(b)$, we have the open trek $a\to d_1\to\dots\to b$. Similarly to the previous case, if we conditioned on either $d_1$ or any vertex in $\de(d_1)$, then we have the open trek $a\to d_1\leftrightarrow d_2\to\dots\to b$ and so on until the trek $a\to d_1\leftrightarrow \dots\leftrightarrow b$ is open. Hence, $\lambda_{ab}$ cannot be identified by an ordinary regression.
\end{enumerate}
\end{proof}

Now, we have a necessary and sufficient criterion for $\Sigma^*$ to be simply identifiable given $\Sigma$.

\begin{cor}
\label{cor: remove-directed}
Suppose that $N(0,\Sigma)$ is a distribution corresponding to the graphical model $G$ and the distribution $N(0,\Sigma^*)$ corresponds to the new graphical model $G^*$ obtained by removing the directed edge $a^*\to b$ from $G$. Then, $\Sigma^*$ is simply identifiable given $\Sigma$ if and only if
\begin{itemize}
    \item $a^*\notin\mb_{\an^*(b)}(b)$, where $\mb_{\an^*(b)}(b)$ denotes the Markov blanket of $b$ in $G^*_{\an(b)}$.
    \item $b$ is fixable.
\end{itemize}
\end{cor}

\subsection{Effect on Data}
We will now explore the effects that adding a directed edge will have on specific data point assuming multivariate Gaussian distribution and that the error terms are stable. This is highly practical as finding the counterfactual distribution allows us to evaluate the impact of various policy, business and medical decisions before their implementation. This could answer crucial questions such as ``how much would vaccination reduce the risk of Covid" or ``will undergoing a particular treatment increase the probability of survival" for a particular individual.

Suppose we have a dataset $X_V\sim N(0,\Sigma)$ and we would like to remove the directed edge $a^*\to b$. 
\begin{thm}
\label{thm: remove-directed-data}
Suppose $\lambda_{a^*b}$ and $\sigma(\mathcal{D}_{bc})$ are identifiable for all $c\in\de(b)\backslash\{b\}$, then removing the directed edge $a^*\to b$ is equivalent to replacing $X_V$ by
$$X_V^*=X_V-\lambda_{a^*b}\Sigma(\mathcal{D}_{bV})X_{a^*}.$$
\end{thm}

\begin{proof}
Since $a\in\nd(b)$, $\sigma(\mathcal{D}_{ba})=0$. Hence. $X_A^*=X_A$. Similarly, since $\sigma(\mathcal{D}_{bb})=1$, we can check that for all $a\in A$ and $c\in C$,
\begin{align*}
\Cov(X^*_a, X^*_b) &= \Cov(X_a, X_b) - \lambda_{a^*b} \Cov(X_a, X_{a^*})\\
&= \sigma_{ab}-\lambda_{a^*b}\sigma_{a^*a},\\
\Cov(X^*_b, X^*_b) &= \Cov(X_b, X_b) - 2\lambda_{a^*b} \Cov(X_{a^*}, X_b)+\lambda^2_{a^*b}\Cov(X_{a^*}, X_b)\\
&= \sigma_{bb}-2\lambda_{a^*b}\sigma_{a^*b}+\lambda^2_{a^*b}\sigma_{a^*a^*},\\
\Cov(X^*_a, X^*_c) &= \Cov(X_a, X_c) - \lambda_{a^*b}\sigma(\mathcal{D}_{bc})\Cov(X_a, X_{a^*})\\
&= \sigma_{ac}-\lambda_{a^*b}\sigma(\mathcal{D}_{bc})\sigma_{a^*a},
\end{align*}
which satisfy the Corollaries found in Section \ref{section: directed-covariance}. Similarly, we can check that $\Cov(X^*_b, X^*_c)$ and $\Cov(X^*_c, X^*_{c'})$ satisfies the Corollaries in Section \ref{section: directed-covariance}. 
\end{proof}



\section{Adding a Directed Edge}
Instead of removing an edge, we now consider the reverse intervention whereby an edge is added to an ADMG $G$. In this scenario, we have our own choice for the edge coefficients $\lambda_{a^*b}$ and hence there is no need to identify it.

Once again, suppose that we have an ADMG $G=(V,\D,\B)$ with covariance matrix $\Sigma$ and $a^*\to b\notin\D$. This time, we want to add the directed edge $a^*\to b$ to $G$. As in Section \ref{section: remove-directed-edge}, let $A=\nd(b)$, $B=\{b\}$ and $C=\de(b)\backslash\{b\}$ and suppose we obtain a new ADMG $G^*$ with covariance matrix $\Sigma^*$ by adding the directed edge $a^*\to b$ with edge coefficient $\lambda_{a^*b}$. By the trek rule, we have
\begin{align*}
    \Sigma^*_{AA}&=\Sigma_{AA},\\
    \Sigma^*_{Ab}&=\Sigma_{Ab}+\lambda_{a^*b}\Sigma_{Aa^*},\\
    \sigma^*_{bb}&=\sigma_{bb}+2\lambda_{a^*b}\sigma_{a^*b}-\lambda^2_{a^*b}\sigma_{a^*a^*},\\
    \Sigma_{AC}^*&=\Sigma_{AC}+\lambda_{a^*b}\Sigma_{Aa^*}\Sigma(\mathcal{D}_{bC}),\stepcounter{equation}\tag{\theequation}\label{eqn: add-directed}\\
    \Sigma^*_{bC}&=\Sigma_{bC}+\lambda_{a^*b}\Sigma_{a^*C}+\lambda_{a^*b}\sigma_{a^*b}\Sigma(\mathcal{D}_{bC})-\lambda_{a^*b}^2\sigma_{a^*a^*}\Sigma(\mathcal{D}_{bC}),\\
    \Sigma^*_{CC}&=\Sigma_{CC}+\lambda_{a^*b}\Sigma_{a^*C}^T\Sigma(\mathcal{D}_{bC})+\lambda_{a^*b}\Sigma_{a^*C}^T\Sigma(\mathcal{D}_{bC})-\lambda_{a^*b}^2\sigma_{a^*a^*}\Sigma^T(\mathcal{D}_{bC})\Sigma(\mathcal{D}_{bC}),
\end{align*}
where $\Sigma(\mathcal{D}_{bC})=[\sigma(\mathcal{D}_{bc})]_{c\in C}$ is the vector of directed treks.
Note that the equations are identical to those in Section \ref{section: directed-covariance}, with the signs flipped. Hence, in order to add the directed edge $a^*\to b$ with edge coefficient $\lambda_{a^*b}$, we need to be able to identify $\Sigma(\mathcal{D}_{bC})$. 

\begin{prop}
\label{prop: add-directed}
Suppose that the covariance matrix $\Sigma^*$ is obtained by adding a directed edge $a^*\to b$ to $G$. Then, $\Sigma^*$ is identifiable given $\Sigma$ if and only if $\sigma(\mathcal{D}_{bc})$ is identifiable for all $c\in C$.
\end{prop}

\begin{proof}
Rearrange (\ref{eqn: add-directed}) to obtain $\Sigma(\mathcal{D}_{bC})$ as a function of $\Sigma$, $\Sigma^*$ and $\lambda_{a^*b}$.
\end{proof}

\begin{dfn}
Suppose the covariance matrix $\Sigma^*$ is obtained from $\Sigma$ by the addition of the directed edge $a^*\leftrightarrow b$. We say that $\Sigma^*$ is \emph{simply identifiable} given $\Sigma$ if $\sigma(\mathcal{D}_{bc})$ is regression identifiable for all $c\in C$.
\end{dfn}

By Theorem \ref{thm: fixable-paths}, we have the following Corollary.

\begin{cor}
\label{cor: add-directed}
Suppose that we obtain an ADMG $G^*$ with covariance matrix $\Sigma^*$ by adding a directed edge $a^*\to b$ with edge coefficient $\lambda_{a^*b}$ to $G$. Then, $\Sigma^*$ is simply identifiable given $\Sigma$ if and only if $b$ is fixable.
\end{cor}

\begin{thm}
\label{thm: add-directed-data}
Suppose $\lambda_{a^*b}$ and $\sigma(\mathcal{D}_{bc})$ are identifiable for all $c\in\de(b)\backslash\{b\}$, then adding the directed edge $a^*\to b$ is equivalent to replacing $X_V$ by
$$X_V^*=X_V+\lambda_{a^*b}\Sigma(\mathcal{D}_{bV})X_{a^*}.$$
\end{thm}

\begin{proof}
Similar to the proof of Theorem \ref{thm: remove-directed-data}.
\end{proof}

\section{Limitations}
\label{section: limitations}
For our main results on edge interventions, we relied heavily on simple identifiability and assumed that the edge coefficients $\lambda_{ab}$ and path coefficients $\sigma(\mathcal{D}_{bc})$ are regression identifiable. However, for $\Sigma^*$ to be identifiable given $\Sigma$, we only require $\lambda_{ab}$ and $\sigma(\mathcal{D}_{bc})$ to be generically identifiable. In this section, we shall provide some other graphical criteria for which $\lambda_{ab}$ or $\sigma(\mathcal{D}_{bc})$ is generically identifiable but not regression identifiable.

\subsection{Edge Identifiability}
In this section, we shall provide two other methods for finding sufficient conditions for $\lambda_{ab}$ to be generically identifiable. These methods are not exhaustive as the question of finding necessary and sufficient conditions for $\lambda_{ab}$ to be generically identifiable remains an open problem.

\subsubsection*{Simple Graphs}

A mixed graph $G=(V,\D,\B)$ is simple if between any two vertices $v_1,v_2\in V$, there is at most one edge between them. That is, we cannot have both $v_1\to v_2$ and $v_1\leftrightarrow v_2$. 

\begin{thm}
If $G$ is simple, then $\M(G)$ is generically identifiable.
\end{thm}

\begin{proof}
See \citet{brito2002new}.
\end{proof}

Using the Theorem above, we can obtain a similar result to Corollary \ref{cor: remove-directed}:

\begin{prop}
Suppose that $N(0,\Sigma)$ is a distribution corresponding to the graphical model $G$ and the distribution $N(0,\Sigma^*)$ corresponds to the new graphical model $G^*$ obtained by removing the directed edge $a^*\to b$ from $G$. Then, $\Sigma^*$ is generically identifiable given $\Sigma$ if
\begin{itemize}
    \item $G$ is a simple graph.
    \item $b$ is fixable.
\end{itemize}
\end{prop}

\subsubsection*{Generalised Instrumental Sets}

In the instrumental variable model shown in Figure \ref{fig: IV}, the directed edge $\lambda_{23}$ generically identifiable even though it is not regression identifiable. This lead to the study of an \emph{instrumental set} by \citet{brito2012generalized}, where the instrumental variable model was generalised.


\begin{thm}
\label{thm: generalised-IV}
If $\mathbf{z}=\{z_1,\dots,z_n\}$ is an instrumental set relative to causes $\mathbf{a}=\{a_1,\dots,a_n\}$ and effect $b$, then the parameters of the edges $a_1\to b,\dots,a_n\to b$ are identified almost everywhere, and can be computed by solving a system of linear equations.
\end{thm}

If we are only interested in identifying the parameter of the edge $a^*\to b$, we only need the existence of an instrumental variable $z$ that is instrumental relative to $a^*$ and $b$. Hence, 

\begin{cor}
Suppose that $N(0,\Sigma)$ is a distribution corresponding to the graphical model $G$ and the distribution $N(0,\Sigma^*)$ corresponds to the new graphical model $G^*$ obtained by removing the directed edge $a^*\to b$ from $G$. Then, $\Sigma^*$ is generically identifiable given $\Sigma$ if 
\begin{itemize}
    \item If there exists an vertex $z$ that is instrumental relative to $a^*$ and $b$.
    \item $b$ is fixable.
\end{itemize}
\end{cor}

\subsection{Cut Vertices}
Previously, we dealt with cases where the edge coefficients are generically identifiable but not regression identifiable. Now, we shall provide some cases where directed path coefficients are generically identifiable but not regression identifiable.

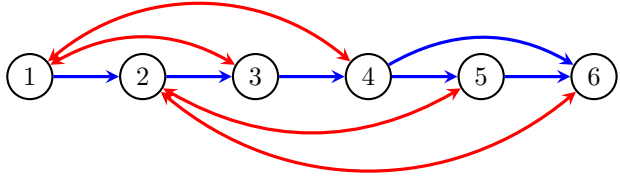
\begin{figure}
  \begin{center}
  \begin{tikzpicture}
  [rv/.style={circle, draw, thick, minimum size=6mm, inner sep=0.5mm}, node distance=15mm, >=stealth]
  \pgfsetarrows{latex-latex};
\begin{scope}
  \node[rv]  (1)              {1};
  \node[rv, right of=1, yshift=0mm, xshift=0mm] (2) {2};
  \node[rv, right of=2, yshift=0mm, xshift=0mm] (3) {3};
  \node[rv, right of=3] (4) {4};
  \node[rv, right of=4, yshift=0mm, xshift=0mm] (5) {5};
  \node[rv, right of=5, yshift=0mm, xshift=0mm] (6) {6};
  \draw[->, very thick, color=blue] (1) -- (2);
  \draw[->, very thick, color=blue] (2) -- (3);
  \draw[->, very thick, color=blue] (3) -- (4);
  \draw[->, very thick, color=blue] (4) -- (5);
  \draw[->, very thick, color=blue] (5) -- (6);
  \draw[->, very thick, color=blue] (4) to[bend left=30] (6);
  \draw[<->, very thick, color=red] (1) to[bend left=40] (4);
  \draw[<->, very thick, color=red] (1) to[bend left=30] (3);
  \draw[<->, very thick, color=red] (2) to[bend right=30] (5);
  \draw[<->, very thick, color=red] (2) to[bend right=40] (6);
  \end{scope}
    \end{tikzpicture}
  \end{center}
  \caption{Cut-vertex example}
  \label{fig: hard ADMG}
\end{figure}

\begin{ex}
\label{ex: cut-vertex}
Consider the graph in Figure \ref{fig: hard ADMG}. Here, we cannot use Corollary \ref{cor: remove-directed} to remove the edge $1\to 2$ as 2 is not fixable. However, we can observe that all directed paths contributing to $\sigma(\mathcal{D}_{25})$ and $\sigma(\mathcal{D}_{26})$ pass through the vertex 4. Hence, we are able to identify those using $\sigma(\mathcal{D}_{25})=\sigma(\mathcal{D}_{24})\sigma(\mathcal{D}_{45})=\beta_{24\cdot 1}\beta_{45\cdot123}$ and $\sigma(\mathcal{D}_{26})=\sigma(\mathcal{D}_{24})\sigma(\mathcal{D}_{46})=\beta_{24\cdot 1}\beta_{46\cdot123}$. In fact, we see that it is possible to remove the edge $1\to 2$ and identify the new covariance matrix $\Sigma^*$. We shall refer to the vertex 4 as a \emph{cut vertex} which we shall rigorously define.
\end{ex}


\begin{dfn}
For a CADMG, $G$, let $G^{(d)}$ be the subgraph of $G$ with all the bidirected edges removed and let $G^{(d)}_{a\to b}$ be the induced subgraph $G_{\an(b)\cap\de(a)}$ with all bidirected edges removed. We say that a vertex $v$ is a \emph{cut vertex} of $G^{(d)}_{a\to b}$ if $G^{(d)}_{a\to b}-\{v\}$ is disconnected.
\end{dfn}

\begin{lem}
\label{lem-split-paths}
Let $b\in V$ and suppose $c\in\de(b)$. Denote the directed paths from $b$ to $c$ by $\tau_1,\dots,\tau_n$. Further partition each $\tau_i$ into directed subpaths $\tau_i^1,\dots,\tau_i^m$. Then $$\sigma(\mathcal{D}_{bc})=\sum\limits_{i=1}^n\prod\limits_{j=1}^m\sigma(\tau_i^j).$$
In particular, if $u_1,\dots,u_n$ is a sequence of cut-vertices in $G^{(d)}_{b\to c}$, then 
$$\sigma(\mathcal{D}_{bc})=\sigma(\mathcal{D}_{bu_1})\sigma(\mathcal{D}_{u_1u_2})\dots\sigma(\mathcal{D}_{u_nc}).$$
\end{lem}

\begin{proof}
Follows from the trek rule.
\end{proof}

\begin{cor}
Suppose that $N(0,\Sigma)$ is a distribution corresponding to the graphical model $G$ and the distribution $N(0,\Sigma^*)$ corresponds to the new graphical model $G^*$ obtained by removing the directed edge $a^*\to b$ from $G$. Then, $\Sigma^*$ is identifiable given $\Sigma$ if 
\begin{itemize}
    \item $a^*\notin\mb_{\an^*(b)}(b)$, where $\mb_{\an^*(b)}(b)$ denotes the Markov blanket of $b$ in $G^*_{\an(b)}$.
    \item For all $c\in C$, we can a sequence of vertices $v_0,v_1,\dots,v_n$ such that $v_0=b$, $v_n=c$, $v_1,\dots,v_{n-1}$ is a sequence of cut-vertices in $G^{(d)}_{b\to c}$ and $v_i$ is fixable in $G_{\an(v_{i+1})}$ for all $0\leq i\leq n-1$.
\end{itemize}
\end{cor}

\section{Usage of Removing Edges}
In this section, we shall showcase some statistical problems where our identifiability results could be applied. We first show that our results may be used to find constraints in ADMGs, provided that we can either remove all the directed or bidirected edges (see Appendix \ref{appendix: bidirected}). We shall then provide an example from epidemiology where we have two treatments, and we make an intervention on conditions required for the second treatment. 
\subsection{Constraints}
\begin{figure}
\centering
  \begin{tikzpicture}
  [rv/.style={circle, draw, thick, minimum size=6mm, inner sep=0.5mm}, node distance=15mm, >=stealth,
  hv/.style={circle, draw, thick, dashed, minimum size=6mm, inner sep=0.5mm}, node distance=15mm, >=stealth]
  \pgfsetarrows{latex-latex};
  \begin{scope}
  \node[rv]  (1)              {1};
  \node[rv, right of=1, yshift=0mm, xshift=0mm] (2) {2};
  \node[rv, right of=2, yshift=0mm, xshift=0mm] (3) {3};
  \node[rv, right of=3] (4) {4};
  \draw[->, very thick, color=blue] (1) -- (2);
  \draw[->, very thick, color=blue] (2) -- (3);
  \draw[->, very thick, color=blue] (3) -- (4);
  \draw[->, very thick, color=blue] (1) to[bend right] (3);
  \draw[<->, very thick, color=red] (2) to[bend left] (4);
  \end{scope}
    \end{tikzpicture}
 \caption{The Verma graph.}
  \label{fig: Verma}
\end{figure}
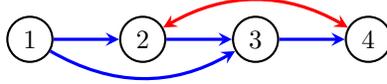
In a DAG $G$, the set of conditional independences yields an implicit description of $\M(G)$ \citep{garcia2005algebraic, kiiveri1984recursive}. However, in the class of ADMGs, \citet{robins1986new} noted a new form of constraint similar to that of conditional independence, sometimes referred to as the \emph{functional constraints}.

Consider the Verma graph in Figure \ref{fig: Verma}. Here, there are no conditional independences involving only the observed variables $X_1, X_2, X_3$ and $X_4$. However, we do have a constraint on the corresponding covariance matrix $\Sigma$, in the sense that $\Sigma\in\M(G)$, where $\M(G)$ denotes the set of covariance matrices that can exist for a given ADMG $G$, only if
\begin{align*}
f_{\mathrm{Verma}}(\Sigma)&=\sigma_{11}\sigma_{13}\sigma_{22}\sigma_{34}-\sigma_{12}^2\sigma_{13}\sigma_{34}-\sigma_{11}\sigma_{14}\sigma_{22}\sigma_{33}+\sigma_{12}^2\sigma_{14}\sigma_{33}\\ &\qquad-\sigma_{11}\sigma_{13}\sigma_{23}\sigma_{24}+\sigma_{11}\sigma_{14}\sigma_{23}^2+\sigma_{12}\sigma_{13}^2\sigma_{24}-\sigma_{12}\sigma_{13}\sigma_{14}\sigma_{23}=0.
\end{align*}
This polynomial is commonly referred to as the `Verma constraint' courtesy of \citet{verma1991equivalence}.
In non-parametric models, \citet{tian2002testable} created an algorithm for identifying functional constraints. For instance, the Verma constraint can be seen as the independence between $X_1$ and $X_4$ after fixing $X_2$ and $X_3$ (i.e. after removing all edges pointing into vertices 2 and 3). Tian's algorithm, though non-parametrically complete, will fail to find the Gaussian constraint on the `gadget' graph in Figure \ref{fig: gadget}:

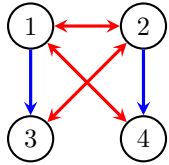
\begin{figure}
  \begin{center}
  \begin{tikzpicture}
  [rv/.style={circle, draw, thick, minimum size=6mm, inner sep=0.5mm}, node distance=15mm, >=stealth,
  hv/.style={circle, draw, thick, dashed, minimum size=6mm, inner sep=0.5mm}, node distance=15mm, >=stealth]
  \pgfsetarrows{latex-latex};
  \begin{scope}
  \node[rv]  (1)              {1};
  \node[rv, right of=1, yshift=0mm, xshift=0mm] (2) {2};
  \node[rv, below of=1, yshift=0mm, xshift=0mm] (3) {3};
  \node[rv, below of=2, yshift=0mm, xshift=0mm] (4) {4};
  \draw[->, very thick, color=blue] (1) -- (3);
  \draw[->, very thick, color=blue] (2) -- (4);
  \draw[<->, very thick, color=red] (1) -- (2);
  \draw[<->, very thick, color=red] (1) -- (4);
  \draw[<->, very thick, color=red] (2) -- (3);
  \end{scope}
    \end{tikzpicture}
 \caption{The `gadget' graph.}
 \label{fig: gadget}
  \end{center}
\end{figure}
$$\sigma_{11}\sigma_{22}\sigma_{34}-\sigma_{13}\sigma_{14}\sigma_{22}+\sigma_{13}\sigma_{12}\sigma_{24}-\sigma_{23}\sigma_{11}\sigma_{24}=0.$$

However, for the parametric model, the set of polynomials describing $\M(G)$ for an arbitrary ADMG $G$ remains an open problem despite various attempts, such as that of \citet{drton2018nested}. In this section, we will demonstrate how removing edges could allow us to recover the constraints of an ADMG.

Suppose we can iteratively remove bidirected or directed edges from an ADMG $G$ to obtain a new graph $G^*$ with covariance matrix $\Sigma^*$. Then, the constraints implied by $G$ correspond to the non-trivial polynomials of the form $\sigma^*_{ij}$ for all vertex pairs $i,j$ without a trek between them in $G^*$. If $G^*$ has no directed edges, this will simply correspond to non-adjacent vertex pairs. If $G^*$ has no edges left, this will correspond to the strictly upper triangular entries of $\Sigma^*$. We shall provide examples where we remove all directed edges.

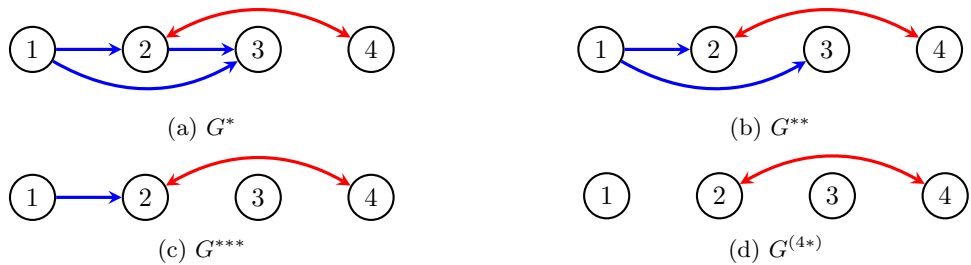
\begin{figure}[ht] 
  \begin{subfigure}[b]{0.5\linewidth}
    \centering
  \begin{tikzpicture}
  [rv/.style={circle, draw, thick, minimum size=6mm, inner sep=0.5mm}, node distance=15mm, >=stealth,
  hv/.style={circle, draw, thick, dashed, minimum size=6mm, inner sep=0.5mm}, node distance=15mm, >=stealth]
  \pgfsetarrows{latex-latex};
  \begin{scope}
  \node[rv]  (1)              {1};
  \node[rv, right of=1, yshift=0mm, xshift=0mm] (2) {2};
  \node[rv, right of=2, yshift=0mm, xshift=0mm] (3) {3};
  \node[rv, right of=3] (4) {4};
  \draw[->, very thick, color=blue] (1) -- (2);
  \draw[->, very thick, color=blue] (2) -- (3);
  \draw[->, very thick, color=blue] (1) to[bend right] (3);
  \draw[<->, very thick, color=red] (2) to[bend left] (4);
  \end{scope}
    \end{tikzpicture}
 \caption{$G^*$}
 \label{fig: G*}
  \end{subfigure}
  \begin{subfigure}[b]{0.5\linewidth}
  \centering
  \begin{tikzpicture}
  [rv/.style={circle, draw, thick, minimum size=6mm, inner sep=0.5mm}, node distance=15mm, >=stealth,
  hv/.style={circle, draw, thick, dashed, minimum size=6mm, inner sep=0.5mm}, node distance=15mm, >=stealth]
  \pgfsetarrows{latex-latex};
  \begin{scope}
  \node[rv]  (1)              {1};
  \node[rv, right of=1, yshift=0mm, xshift=0mm] (2) {2};
  \node[rv, right of=2, yshift=0mm, xshift=0mm] (3) {3};
  \node[rv, right of=3] (4) {4};
  \draw[->, very thick, color=blue] (1) -- (2);
  \draw[->, very thick, color=blue] (1) to[bend right] (3);
  \draw[<->, very thick, color=red] (2) to[bend left] (4);
  \end{scope}
    \end{tikzpicture}
 \caption{$G^{**}$}
 \label{fig: G**}
  \end{subfigure} 
  \begin{subfigure}[b]{0.5\linewidth}
    \centering
      \begin{tikzpicture}
  [rv/.style={circle, draw, thick, minimum size=6mm, inner sep=0.5mm}, node distance=15mm, >=stealth,
  hv/.style={circle, draw, thick, dashed, minimum size=6mm, inner sep=0.5mm}, node distance=15mm, >=stealth]
  \pgfsetarrows{latex-latex};
  \begin{scope}
  \node[rv]  (1)              {1};
  \node[rv, right of=1, yshift=0mm, xshift=0mm] (2) {2};
  \node[rv, right of=2, yshift=0mm, xshift=0mm] (3) {3};
  \node[rv, right of=3] (4) {4};
  \draw[->, very thick, color=blue] (1) -- (2);
  \draw[<->, very thick, color=red] (2) to[bend left] (4);
  \end{scope}
    \end{tikzpicture}
 \caption{$G^{***}$}
 \label{fig: G***}
  \end{subfigure}
  \begin{subfigure}[b]{0.5\linewidth}
    \centering
  \begin{tikzpicture}
  [rv/.style={circle, draw, thick, minimum size=6mm, inner sep=0.5mm}, node distance=15mm, >=stealth,
  hv/.style={circle, draw, thick, dashed, minimum size=6mm, inner sep=0.5mm}, node distance=15mm, >=stealth]
  \pgfsetarrows{latex-latex};
  \begin{scope}
  \node[rv]  (1)              {1};
  \node[rv, right of=1, yshift=0mm, xshift=0mm] (2) {2};
  \node[rv, right of=2, yshift=0mm, xshift=0mm] (3) {3};
  \node[rv, right of=3] (4) {4};
  \draw[<->, very thick, color=red] (2) to[bend left] (4);
  \end{scope}
    \end{tikzpicture}
 \caption{$G^{(4*)}$}
 \label{fig: G****}
  \end{subfigure} 
  \caption{Illustration of each step of edge removal in Example \ref{ex: Verma-remove}}
  \label{fig: Verma-ex} 
\end{figure}

\begin{ex}[Removing directed edges from Verma graph]
\label{ex: Verma-remove}
Consider the Verma graph in Figure \ref{fig: Verma}. Note that we may not remove the directed edge $1\to 2$ since the vertex 2 is not fixable. Suppose instead we first remove the directed edge $3\to 4$ to obtain a new graph $G^*$. We shall illustrate each step of this example in Figure \ref{fig: Verma-ex}. Then, $A=\{1,2,3\}$, $B=\{4\}$ and $C=\emptyset$. Now, since $\lambda_{34}={\sigma_{34\cdot 12}}/{\sigma_{33\cdot 12}}$ by Theorem \ref{thm: lambda_ab}, the covariance matrix of the new graph, $\Sigma^*$, becomes

$$\Sigma^*=\left[\begin{array}{cc}
     \Sigma_{123,123} & \Sigma_{123,4}-\Sigma_{123,3}\frac{\sigma_{34\cdot 12}}{\sigma_{33\cdot 12}}\\
     & \sigma_{44}-2\frac{\sigma_{34}\sigma_{34\cdot 12}}{\sigma_{33\cdot 12}}+\sigma_{33}\frac{\sigma^2_{34\cdot 12}}{\sigma^2_{33\cdot 12}}
\end{array}\right].$$

In this new graph with the edge $3\to 4$ removed, the vertex 2 is now fixable. Now, removing the edge $2\to 3$, $\lambda^*_{23}=\frac{\sigma_{23\cdot1}}{\sigma_{22\cdot1}}$, $\sigma^*(\mathcal{D}_{34})=0$.
$$\Sigma^{**}=\left[\begin{array}{ccc}
      \Sigma_{12,12} &  \Sigma_{12,3}-\Sigma_{12,2}\frac{\sigma_{23\cdot1}}{\sigma_{22\cdot1}} & \Sigma^{*}_{12,4} \\
      & \sigma_{33}-\frac{2\sigma_{23}\sigma_{23\cdot1}}{\sigma_{22\cdot1}}+\frac{\sigma^2_{23\cdot1}}{\sigma_{22\cdot1}} & \sigma^*_{34}-\frac{\sigma_{23\cdot1}\sigma^*_{24}}{\sigma_{22\cdot1}} \\
      & & \sigma^*_{44}
\end{array}\right].$$

Removing edge $1\to 3$, $\lambda^{**}_{13}=\frac{\sigma^{**}_{13}}{\sigma^{**}_{11}}$, $\sigma^{**}(\mathcal{D}_{34})=0$.

$$\Sigma^{***}=\left[\begin{array}{ccc}
      \Sigma_{12,12} &  \Sigma^{**}_{12,3}-\Sigma_{12,1}\frac{\sigma^{**}_{13}}{\sigma_{11}} & \Sigma^{*}_{12,4} \\
      & \sigma_{33}-2\frac{\sigma^{**^2}_{13}}{\sigma_{11}}+\frac{\sigma^{**^2}_{13}}{\sigma^{**}_{11}} & \sigma^{**}_{34}-\frac{\sigma^{**}_{13}\sigma^{*}_{14}}{\sigma_{11}} \\
      & & \sigma^*_{44}
\end{array}\right].$$

Removing edge $1\to 2$ $\lambda^{***}_{12}=\frac{\sigma_{12}}{\sigma_{11}}$ , $\sigma^{***}(\mathcal{D}_{23})=\sigma^{***}(\mathcal{D}_{24})=0$.
$$\Sigma^{(4*)}=\left[\begin{array}{cccc}
     \sigma_{11} & 0 & \sigma^{***}_{13} & \sigma^*_{14} \\
     & \sigma_{22}-\frac{\sigma^2_{12}}{\sigma_{11}} & \sigma^{***}_{23}-\frac{\sigma_{12}\sigma^{***}_{13}}{\sigma_{11}}& \sigma^*_{24}-\frac{\sigma_{12}\sigma^*_{14}}{\sigma_{11}}\\
     & & \sigma^{***}_{33} & \sigma^{***}_{34}\\
     & & & \sigma^*_{44}
\end{array}\right].$$

Note $\sigma^{***}_{13}=0$ since there are no treks from 1 to 3 in $G^{***}$ (see Figure \ref{fig: G***}). Now, since $\Sigma^{(4*)}$ corresponds to the graph on 4 vertices with only the bidirected edge $2\leftrightarrow 4$, we can equate the strictly upper triangular part of $\Sigma^{(4*)}$, apart from $\sigma^{(4*)}_{24}$ to be zero to obtain the constraints of the original graph.

Since there are no treks from vertex 3 to any of the other vertices in $G^{***}$, $\sigma^{***}_{13}$, $\sigma^{(4*)}_{23}$ and $\sigma^{***}_{34}$ are identically zero by the trek rule. Hence, equating the entries of $\Sigma^{(4*)}$ corresponding to vertex pairs in $G^{(4*)}$ without a trek between them to zero, we have four expressions that trivially vanish and the equation $\sigma^{*}_{14}=0$. Equating $\sigma^{*}_{14}=0$, we recover $f_{\mathrm{Verma}}=0$. We show this computation in full in the appendix.
\end{ex}

\begin{ex}[The `double' Verma graph]
\begin{figure}
\centering
  \begin{tikzpicture}
  [rv/.style={circle, draw, thick, minimum size=6mm, inner sep=0.5mm}, node distance=15mm, >=stealth,
  hv/.style={circle, draw, thick, dashed, minimum size=6mm, inner sep=0.5mm}, node distance=15mm, >=stealth]
  \pgfsetarrows{latex-latex};
  \begin{scope}
  \node[rv]  (0)              {0};
  \node[rv, right of=0, yshift=0mm, xshift=0mm]  (1)              {1};
  \node[rv, right of=1, yshift=0mm, xshift=0mm] (2) {2};
  \node[rv, right of=2, yshift=0mm, xshift=0mm] (3) {3};
  \node[rv, right of=3] (4) {4};
  \draw[->, very thick, color=blue] (1) -- (2);
  \draw[->, very thick, color=blue] (2) -- (3);
  \draw[->, very thick, color=blue] (3) -- (4);
  \draw[->, very thick, color=blue] (1) to[bend right] (3);
  \draw[<->, very thick, color=red] (2) to[bend left] (4);
  \draw[<->, very thick, color=red] (0) to[bend left] (2);
  \draw[<->, very thick, color=red] (0) to[bend left=40] (3);
  \draw[<->, very thick, color=red] (0) to[bend left=50] (4);
  \draw[->, very thick, color=blue] (0) -- (1);
  \end{scope}
    \end{tikzpicture}
 \caption{The `double' Verma graph.}
  \label{fig: Double_Verma}
\end{figure}
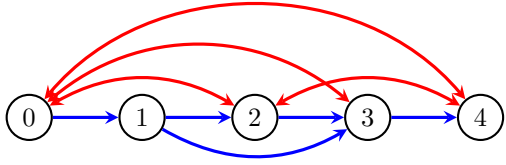

Consider the `double' Verma graph in Figure \ref{fig: Double_Verma} with a doubly nested constraint where we need to fix two vertices, with a marginalisation in between, to obtain the conditional independence constraints. 

In this graph, the vertex 1 is fixable and the vertex 0 is not in the Markov blanket of 1 with the edge $0\to 1$ removed. Hence, we can remove the edge $0\to 1$ to obtain a new ADMG $G^*$. Now, we have a graph similar to the Verma example, where we are able to remove the directed edge $3\to 4$ to obtain a new ADMG $G^{**}$ since 3 is fixable in $G^*_{\an(4)}$. Now, we can remove edges $2\to 3$ and $1\to 3$ as vertices $1$ and 2 are fixable in $G^{**}_{\an(3)}$. Similarly, we are able to remove the edge $1\to 2$. Hence, we are able to recover the constraint in the double Verma graph.
\end{ex}

\subsection{Simulating Interventional Data}

We will now provide a real world data example where we want to simulate an interventional distribution from observational data.

\begin{ex}
We consider the data from the Multicenter AIDS Cohort Study (MACS) \citep{kaslow1987multicenter}, a longitudinal observational study of HIV-infected men. Suppose that we are interested in the effect that treating HIV patients with AZT has on their CD4 cell count in 6 months (1 follow-up period). We restrict ourselves to a subset of the MACS study consisting of 1210 men for which there are no missing data for age, medication, AZT doses (if treated) and CD4 count. The mean (interquartile range) of these men is 43.2 (38--49) years old. The graphical model is provided in Figure \ref{fig: MACS} with variables:
\begin{figure}
    \centering
      \begin{tikzpicture}
  [rv/.style={circle, draw, thick, minimum size=6mm, inner sep=0.5mm}, node distance=15mm, >=stealth,
  hv/.style={circle, draw, thick, dashed, minimum size=6mm, inner sep=0.5mm}, node distance=15mm, >=stealth]
  \pgfsetarrows{latex-latex};
  \begin{scope}
  \node[rv]  (1)              {1};
  \node[rv, right of=1, yshift=0mm, xshift=0mm] (2) {2};
  \node[rv, right of=2, yshift=0mm, xshift=0mm] (3) {3};
  \node[rv, above of=2, yshift=-2mm, xshift=-7mm] (4) {4};
  \node[rv, right of=4, yshift=0mm, xshift=0mm] (5) {5};
  \draw[->, very thick, color=blue] (1) -- (2);
  \draw[->, very thick, color=blue] (2) -- (3);
  \draw[->, very thick, color=blue] (1) -- (4);
  \draw[->, very thick, color=blue] (2) -- (5);
  \draw[->, very thick, color=blue] (4) -- (2);
  \draw[->, very thick, color=blue] (5) -- (3);
  \draw[<->, very thick, color=red] (1) to[bend right=45] (3);
  \draw[<->, very thick, color=red] (1) to[bend right] (2);
  \draw[<->, very thick, color=red] (2) to[bend right] (3);
  \end{scope}
    \end{tikzpicture}
    \caption{Graphical model for MACS study}
    \label{fig: MACS}
\end{figure}
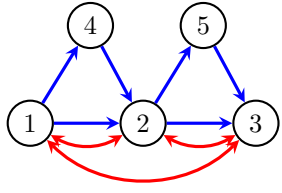
\begin{itemize}
    \item $X_1$ - Baseline CD4 cell count;
    \item $X_2$ - CD4 cell count after first treatment;
    \item $X_3$ - CD4 cell count after second treatment;
    \item $X_4$ - AZT dosage of first treatment (log scale with 1 added);
    \item $X_5$ - AZT dosage of second treatment (log scale with 1 added).
\end{itemize}
In particular, we hypothesised that the patients age will affect both the treatment proposed and the change in CD4 cell count given any antiviral treatment. Standardising the dataset and assuming that our dataset follows a linear SEM, we obtain $\lambda_{14}=-0.0252$ $\lambda_{42}=-0.0176$, $\lambda_{25}=-0.0287$, $\lambda_{53}=-0.0378$, $\sigma(\mathcal{D}_{42})=-0.0387$, $\sigma(\mathcal{D}_{43})=-0.0580$ and $\sigma(\mathcal{D}_{45})=0.891$ with
$$\Sigma=\left[\begin{array}{ccccc}
     1.000 & 0.835 & 0.794 & -0.025 & -0.004  \\
     0.835 & 1.000 & 0.853 & -0.039 & -0.029 \\
     0.794 & 0.853 & 1.000 & -0.058 & -0.051 \\
     -0.025 & -0.039 & -0.058 & 1.000 & 0.891 \\
     -0.004 & -0.029 & -0.051 & 0.891 & 1.000
\end{array}\right].$$
Suppose we want to simulate the dataset for which we no longer determine the AZT dosage of the first treatment based on a patient's baseline CD4 cell count. Performing an edge intervention by removing $\lambda_{14}$, we have $A=\{1\},$ $B=\{4\}$ and $C=\{2,3,5\}$ with
$$\Sigma^*=\left[\begin{array}{ccccc}
     1.000 & 0.835 & 0.793 & -0.025 & 0.018 \\
     0.835 & 0.999 & 0.852 & -0.018 & -0.010\\
     0.793 & 0.852 & 0.999 & -0.038 & -0.033 \\
     -0.025 & -0.018 & -0.038 & 0.999 & 0.891 \\
     0.018 & -0.010 & -0.033 & 0.891 & 1.000
\end{array}\right].$$
We could then simulate the distribution of an edge intervention using $N(0,\Sigma^*)$ and reversing the data standardisation. 

\end{ex}

\section{Conclusion}

In this paper, we studied the effects that removing or adding an edge will have on the covariance matrix $\Sigma$ in a Gaussian ADMG. 

We first defined a identification criterion, stronger than global identifiability, called \emph{regression identifiablity}. The necessary and sufficient graphical conditions for $\lambda_{ab}$, $\sigma(\mathcal{D}_{bc})$ and $\omega_{ab}$ to be regression identifiable are given in Theorems  \ref{thm: fixable-paths}, \ref{thm: lambda_ab} and \ref{thm: omega_ab} respectively. Since regression identifiability implies generic identifiablity, we now have sufficient conditions for $\Sigma^*$ to be identifiable given $\Sigma$, where $\Sigma^*$ is the covariance matrix of the modified graph. We have also identified the effect of an edge intervention on an individual data point in Theorems \ref{thm: remove-directed-data} and \ref{thm: add-directed-data}.

However, even if either $\lambda_{ab}$ or $\sigma(\mathcal{D}_{bc})$ is not regression identifiable, it may still be generically identifiable. We studied some cases of the limitations of our results in Section \ref{section: limitations}. If the graphical conditions that are necessary and sufficient for $\lambda_{ab}$ and $\sigma(\mathcal{D}_{bc})$ to be generically identifiable were to be found, then we would have the graphical conditions which are necessary and sufficient for $\Sigma^*$ to be identifiable given $\Sigma$.

Finally, we have provided some statistical problems, such as finding constraints and simulating interventional data from observation data, where our results can be used.

\bibliographystyle{abbrvnat}
\bibliography{references} 
\appendix

\section{Calculations for Example \ref{ex: Verma-remove}}

In this section, we will complete the working for Example \ref{ex: Verma-remove} and show that $\sigma_{14}^*=f_{\mathrm{Verma}}$. Starting off, we have

\begin{align*}
    \sigma_{14}^*&=0\\
    \sigma_{14}-\frac{\sigma_{13}\sigma_{34\cdot 12}}{\sigma_{33\cdot 12}}&=0\\
    \sigma_{14}\sigma_{33\cdot 12}-\sigma_{13}\sigma_{34\cdot 12}&=0\stepcounter{equation}\tag{\theequation}\label{eqn: appendix1}
\end{align*}

where

\begin{align}
\label{eqn: appendix2}
    \sigma_{33\cdot 12}&=\sigma_{33}-\frac{\sigma_{13}^2\sigma_{22}-\sigma_{12}\sigma_{13}\sigma_{23}+\sigma_{11}\sigma_{23}^2-\sigma_{12}\sigma_{13}\sigma_{23}}{\sigma_{11}\sigma_{22}-\sigma_{12}^2}\\
    \label{eqn: appendix3}
    \sigma_{34\cdot 12}&=\sigma_{34}-\frac{\sigma_{13}\sigma_{14}\sigma_{22}-\sigma_{12}\sigma_{14}\sigma_{23}+\sigma_{11}\sigma_{23}\sigma_{24}-\sigma_{12}\sigma_{13}\sigma_{24}}{\sigma_{11}\sigma_{22}-\sigma_{12}^2}.
\end{align}

Substituting (\ref{eqn: appendix2}) and (\ref{eqn: appendix3}) into (\ref{eqn: appendix1}) and clearing denominators, we obtain

\begin{align*}
    &\sigma_{14}[\sigma_{33}(\sigma_{11}\sigma_{22}-\sigma_{12}^2)-(\sigma_{13}^2\sigma_{22}-\sigma_{12}\sigma_{13}\sigma_{23}+\sigma_{11}\sigma_{23}^2-\sigma_{12}\sigma_{13}\sigma_{23})]-\\
    &\sigma_{13}[\sigma_{34}(\sigma_{11}\sigma_{22}-\sigma_{12}^2)-(\sigma_{13}\sigma_{14}\sigma_{22}-\sigma_{12}\sigma_{14}\sigma_{23}+\sigma_{11}\sigma_{23}\sigma_{24}-\sigma_{12}\sigma_{13}\sigma_{24})]
\end{align*}
which simplifies to $f_{\mathrm{Verma}}$.

\section{Removing a Bidirected Edge}
\label{appendix: bidirected}
Let us consider what would occur in theory if we want to remove a bidirected edge.
Once again, suppose that $G=(V,\D,\B)$ is an ADMG with a corresponding distribution $N(0,\Sigma)$ and we want to remove the bidirected edge $a^*\leftrightarrow b$. Assume without loss of generality that $a^*\in\nd(b)$. Again, let $A=\nd(b)$, $B=\{b\}$ and $C=\de(b)\backslash\{b\}$, then $A\cup B\cup C=V$. Throughout this section, we suppose that after removing the bidirected edge $a^*\leftrightarrow b$, we obtain a new ADMG $G^*$ with a corresponding distribution $N(0,\Sigma^*)$ with the assumption that $\Sigma^*_{AA}=\Sigma_{AA}$.

\subsection{Effects on Covariance}
\label{section: bidirected-covariance}
As before, we shall first compute $\Sigma^*$ given $\Sigma$. The majority of the proofs here are shorten as they are similar to the directed edge case.

\begin{lem}
\label{lem: remove-bi1}
For all $a\in A$, we have
\begin{align*}
    \sigma_{ab}^*&=\sigma_{ab}-\sigma(\mathcal{D}_{a^*a})\omega_{a^*b},\\
    \sigma^*_{bb}&=\sigma_{bb}-\sigma(\mathcal{D}_{a^*b})\omega_{a^*b},\\
    \Sigma^*_{AA}&=\Sigma_{AA}.
\end{align*}
\end{lem}
\begin{proof}
Since $a\in\nd(b)$, every trek from $a$ to $b$ containing the bidirected edge $a^*\leftrightarrow b$ must end with $a^*\leftrightarrow b$. By the same reasoning, all treks from $b$ to $b$ containing the bidirected edge $a^*\leftrightarrow b$ must either start with (or by symmetry end with) $a^*\leftrightarrow b$. Furthermore, there cannot be any treks from $a$ to another $a'\in A$ containing the edge $a^*\leftrightarrow b$. 
\end{proof}

\begin{lem}
\label{lem: remove-bi2}
For all $a\in A$ and $c,c'\in C$, 
\begin{align*}
\sigma_{ac}^*&=\sigma_{ac}-\sigma(\mathcal{D}_{a^*a})\omega_{a^*b}\sigma(\mathcal{D}_{bc}),\\
\sigma_{bc}^*&=\sigma_{bc}-\omega_{a^*b}\sigma(\mathcal{D}_{a^*c})-\sigma(\mathcal{D}_{a^*b})\omega_{a^*b}\sigma(\mathcal{D}_{bc}),\\
\sigma_{cc'}^*&=\sigma_{cc'}-\sigma(\mathcal{D}_{a^*c})\omega_{a^*b}\sigma(\mathcal{D}_{bc'})-\sigma(\mathcal{D}_{a^*c'})\omega_{a^*b}\sigma(\mathcal{D}_{bc}).
\end{align*} 
\end{lem}
\begin{proof}
Treks from $a$ to $c$ containing $a^*\leftrightarrow b$ are of the form $$a\longleftarrow\dots\longleftarrow a^*\longleftrightarrow b\longrightarrow\dots\longrightarrow c.$$
Treks from $b$ to $c$ containing $a^*\leftrightarrow b$ are of the form $$b\longleftarrow\dots\longleftarrow a^*\longleftrightarrow b\longrightarrow\dots\longrightarrow c\quad\text{or}\quad b\longleftrightarrow a^*\longrightarrow\dots\longrightarrow c.$$
Treks from $c$ to $c'$ containing $a^*\leftrightarrow b$ are of the form $$c\longleftarrow\dots\longleftarrow a^*\longleftrightarrow b\longrightarrow\dots\longrightarrow c'\quad\text{or}\quad c\longleftarrow\dots\longleftarrow b\longleftrightarrow a^*\longrightarrow\dots\longrightarrow c'.$$
\end{proof}

Once again, we see that each entry of $\Sigma^*$ is a function of $\Sigma$, $\omega_{a^*b}$, $\sigma(\mathcal{D}_{a^*v})$ and $\sigma(\mathcal{D}_{bc})$ for all $v\in V$ and $c\in C$. Hence, we obtain the following Corollary.

\begin{prop}
\label{prop: remove-bi}
Suppose that $N(0,\Sigma)$ is a distribution corresponding to the graphical model $G$ and the distribution $N(0,\Sigma^*)$ corresponds to the new graphical model $G^*$ obtained by removing the bidirected edge $a^*\leftrightarrow b$ from $G$. Then, $\Sigma^*$ is identifiable given $\Sigma$ if and only if $\omega_{a^*b}$, $\sigma(\mathcal{D}_{a^*v})$ and $\sigma(\mathcal{D}_{bc})$ are identifiable for all $v\in V$ and $c\in C$.
\end{prop}

\begin{proof}
\begin{itemize}
    \item[($\Leftarrow$)] Follows from Lemmas \ref{lem: remove-bi1} and \ref{lem: remove-bi2}.
    \item[($\Rightarrow$)] From Lemma \ref{lem: remove-bi1}, we have
\begin{align*}
    \sigma_{a^*b}^*&=\sigma_{a^*b}-\sigma_{a^*a^*}\omega_{a^*b},\\
    \omega_{a^*b}&=\frac{\sigma_{a^*b}-\sigma_{a^*b}^*}{\sigma_{a^*a^*}}.
\end{align*}
Hence, $\omega_{a^*b}$ is identifiable if $\Sigma^*$ is identifiable given $\Sigma$. Rearranging the first two equations in Lemma \ref{lem: remove-bi1}, we obtain
\begin{align*}
    \sigma(\mathcal{D}_{a^*a})&=\frac{\sigma_{ab}^*-\sigma_{ab}}{\omega_{a^*b}},\\
    \sigma(\mathcal{D}_{a^*b})&=\frac{\sigma_{bb}^*-\sigma_{bb}}{\omega_{a^*b}}.
\end{align*}
Since $\omega_{a^*b}$ is identifiable, $\sigma(\mathcal{D}_{a^*a})$ and $\sigma(\mathcal{D}_{a^*b})$ are also identifiable. Now, the first equation of Lemma \ref{lem: remove-bi2} can be rearranged to

$$\sigma(\mathcal{D}_{bc})=\frac{\sigma_{ac}-\sigma_{ac}^*}{\sigma(\mathcal{D}_{a^*a})\omega_{a^*b}}.$$

Since both $\omega_{a^*b}$ and $\sigma(\mathcal{D}_{a^*a})$ are identifiable, so is $\sigma(\mathcal{D}_{bc})$. Finally, the second equation of Lemma \ref{lem: remove-bi2} can be rearranged to

$$\sigma(\mathcal{D}_{a^*c})=\frac{\sigma_{bc}-\sigma_{bc}^*-\sigma(\mathcal{D}_{a^*b})\omega_{a^*b}\sigma(\mathcal{D}_{bc})}{\omega_{a^*b}}.$$

Since $\omega_{a^*b}$, $\sigma(\mathcal{D}_{a^*b})$ and $\sigma(\mathcal{D}_{bc})$ are all identifiable, $\sigma(\mathcal{D}_{a^*c})$ is also identifiable.
\end{itemize}
\end{proof}

Here, we see that a different identifiability criteria for removing bidirected edges is required.

\begin{dfn}
Suppose that $N(0,\Sigma)$ is a distribution corresponding to the graphical model $G$ and the distribution $N(0,\Sigma^*)$ corresponds to the new graphical model $G^*$ obtained by removing the bidirected edge $a^*\leftrightarrow b$ from $G$. We say that $\Sigma^*$ is \emph{simply identifiable} given $\Sigma$ if $\omega_{a^*b}$, $\sigma(\mathcal{D}_{a^*v})$ and $\sigma(\mathcal{D}_{bc})$ are all regression identifiable for all $v\in V$ and $c\in C$.
\end{dfn}

\subsection{Regression Identifiability}

Once again, we shall find the graphical criterion under which $\Sigma^*$ is identifiable. From Theorem \ref{thm: fixable-paths}, we know that $\sigma(\mathcal{D}_{a^*v})$ is identifiable by regression for all $v\in V$ if and only if $a^*$ is fixable in $G$ with $\sigma(\mathcal{D}_{a^*v})=0$ if $v\in\nd(a^*)$, and $\sigma(\mathcal{D}_{bc})$ is identifiable by regression for all $c\in C$ if and only if $b$ is fixable in $G$. Hence, it suffices to find conditions for which the edge coefficient $\omega_{a^*b}$ is identifiable where both $a^*$ and $b$ are fixable.

\begin{thm}
\label{thm: omega_ab}
Suppose we have a bidirected edge $a^*\leftrightarrow b$ where $a^*$ and $b$ are fixable. Then $\omega_{a^*b}$ is identifiable by regression if and only if there are no bidirected paths from $a^*$ to $b$ in $G^*_{\an(b)}$. In particular, $\omega_{a^*b}=\sigma_{a^*b\cdot\xi}$ where $\xi=\mb_{\an(b)}(b)$.
\end{thm}
\begin{proof}
\begin{itemize}
\item[($\Leftarrow$)] It suffices to prove that $\omega_{a^*b}=\sigma_{a^*b\cdot\xi}$ if there are no bidirected paths from $a^*$ to $b$ in $G^*_{\an(b)}$. First, note that we do not have a directed edge $a^*\to b$ since $a^*$ is fixable. Similar to Theorem \ref{thm: lambda_ab}, it is sufficient to consider the induced subgraph $G_{\an(b)}$, as any path containing a vertex $c\notin G_{\an(b)}$ has some collider $c'\notin G_{\an(b)}$.

Now, paths of the form $a^*\dots\to b$ are blocked by $\pa(b)\subset\mb_{{\an(b)}}(b)$. Paths of the form $a^*\dots \to d_1\leftrightarrow\dots\leftrightarrow d_n\leftrightarrow b$ and $a^*\dots\leftarrow d_1\leftrightarrow\dots\leftrightarrow d_n\leftrightarrow b$ are blocked by $\mb_{{\an(b)}}(b)$. So the only paths that are not blocked are the bidirected paths from $a^*$ to $b$. But the only path from $a^*$ to $b$ is the bidirected edge $a^*\leftrightarrow b$.

\item[($\Rightarrow$)] Suppose there is a bidirected path from $a^*$ to $b$ in $G^*_{\an(b)}$, say $a^*\leftrightarrow d_1\leftrightarrow\dots\leftrightarrow d_n\leftrightarrow b$. Since $d_1\in\an(b)$, there is an open path $a^*\leftrightarrow d_1\to\dots\to b$. If we condition on any intermediate vertices on that path, we have another open path $a^*\leftrightarrow d_1\leftrightarrow d_2\to\dots\to b$ since $d_2\in\an(b)$. By iteratively conditioning on intermediate vertices, we will end up with the open path $a^*\leftrightarrow d_1\leftrightarrow\dots\leftrightarrow d_n\leftrightarrow b$.
\end{itemize}
\end{proof}

\begin{cor}
\label{cor: remove-bidirected}
Suppose that $N(0,\Sigma)$ is a distribution corresponding to the graphical model $G$ and the distribution $N(0,\Sigma^*)$ corresponds to the new graphical model $G^*$ obtained by removing the bidirected edge $a^*\leftrightarrow b$ from $G$. Then, $\Sigma^*$ is simply identifiable given $\Sigma$ if and only if
\begin{itemize}
    \item both $a^*$ and $b$ are fixable, and
    \item $a^*\notin\dis_{{\an^*(b)}}(b)$, where $\dis_{\an^*(b)}(b)$ denotes the Markov blanket of $b$ in $G^*_{\an(b)}$.
\end{itemize}
\end{cor}

\subsection{Concerns when Removing Bidirected Edges}

While we could identify the new distribution when removing a bidirected edge theoretically, several issues may arise in practice:
\begin{enumerate}
    \item {\bf Interpretation.} While removing a directed edge $i\to j$ simply implies removing a direct effect of $X_i$ on $X_j$, the removal of a bidirected edge $i\leftrightarrow j$ is slightly more complicated to interpret as we are removing the covariance between the residuals $\epsilon_i$ and $\epsilon_j$, where $\epsilon_\ell$ is the residual obtained when we regress $X_\ell$ on its parents. One such interpretation is the removal of all unobserved confounders of $X_i$ and $X_j$, which may be impractical. In practice, the removal of a bidirected edge is often a byproduct of a vertex intervention.
    \item {\bf Distribution of the new model.} At the start of this section, we have assumed that after removing a bidirected edge $a^*\leftrightarrow b$, the distribution of $X_{a^*}$ remains fixed. This is fine for removing a directed edge $a^*\to b$ as $X_{a^*}$ is not affected by the removal of the causal effect of $X_{a^*}$ on $X_{b}$. However, when removing the bidirected edge $a^*\leftrightarrow b$, the distribution of both $X_{a^*}$ and $X_{b}$ may be affected which would in turn have an impact on their descendants.
    \item {\bf Identifiability.} While we have the theoretical result to recover the new distribution $N(0,\Sigma^*)$ if $\Sigma^*$ is simply identifiable given $\Sigma$, the new covariance matrix $\Sigma^*$ may not be positive definite; for instance, consider the graphical model in Figure \ref{fig: Cov} with a large value for $\omega_{13}$. If we remove $1\leftrightarrow 3$, we are setting $\omega^*_{13}$ to be zero, which may result in a non-positive definite matrix for $\Sigma^*$.
    \item {\bf Effect on Data.} There is no general expression for $X^*_V$ that satisfies $\Cov(X^*_V,X^*_V)=\Sigma^*$, the covariance matrix found in Section \ref{section: bidirected-covariance}. 
    \begin{figure}
\centering
  \begin{tikzpicture}
  [rv/.style={circle, draw, thick, minimum size=6mm, inner sep=0.5mm}, node distance=15mm, >=stealth,
  hv/.style={circle, draw, thick, dashed, minimum size=6mm, inner sep=0.5mm}, node distance=15mm, >=stealth]
  \pgfsetarrows{latex-latex};
  \node[rv]  (1)              {1};
  \node[rv, right of=1, yshift=0mm, xshift=0mm] (2) {2};
  \node[rv, right of=2, yshift=0mm, xshift=0mm] (3) {3};
  \draw[<->, very thick, color=red] (1) -- (2);
  \draw[<->, very thick, color=red] (2) -- (3);
  \draw[<->, very thick, color=red] (1) to[bend left] (3);
     \end{tikzpicture}
 \caption{A graphical model where removing $1\leftrightarrow 3$ may result in a non-positive definite $\Sigma^*$.}
 \label{fig: Cov}
\end{figure}
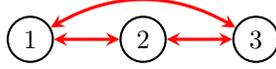
\end{enumerate}
\end{document}